\pdfoutput=1
\documentclass[a4paper]{amsart}

\title{Dimensional reduction and the equivariant Chern character}
\author{Augusto Stoffel}
\address{
  Max Planck Institute for Mathematics\\
  Vivatsgasse 7\\
  53111 Bonn\\
  Germany}
\email{astoffel@mpim-bonn.mpg.de}

\usepackage[utf8]{inputenc}
\usepackage[T1]{fontenc}

\usepackage{amsmath,amsthm,amssymb}
\usepackage{microtype}

\usepackage{hyperref}
\hypersetup{
  pdftitle={\csname @title\endcsname},
  pdfauthor={\authors},
  hidelinks}

\usepackage[hyperref,safeinputenc,maxbibnames=99]{biblatex}
\usepackage[all,cmtip]{xy}

\xymatrixrowsep={1.5em}
\xymatrixcolsep={2.4em}
\objectmargin={0.9ex}

\newtheorem{theorem}{Theorem}
\newtheorem{proposition}[theorem]{Proposition}

\theoremstyle{definition}

\theoremstyle{remark}
\newtheorem{remark}[theorem]{Remark}

\newcommand\abs[1]{\lvert #1\rvert}

\newcommand\Diff{\operatorname{Diff}}

\newcommand\Fun{\operatorname{Fun}}

\newcommand\Isom{\operatorname{Isom}}

\newcommand\ch{\operatorname{ch}}

\newcommand\id{\operatorname{id}}

\newcommand\pr{\operatorname{pr}}
\newcommand\str{\operatorname{str}}
\newcommand\tr{\operatorname{tr}}
\newcommand\colim{\operatorname*{colim}}

\def\mystack#1{\mathrel{\vcenter{\offinterlineskip\ialign{$##$\cr#1\cr}}}}
\newcommand\threerightarrows{\mystack{\to\cr\to\cr\to}}
\newcommand\fourrightarrows{\mystack{\to\cr\to\cr\to\cr\to}}
\newcommand\threeleftarrows{\mystack{\leftarrow\cr\leftarrow\cr\leftarrow}}
\newcommand\fourleftarrows{\mystack{\leftarrow\cr\leftarrow\cr\leftarrow\cr\leftarrow}}

\newcommand\EFT{\mathhyphen\mathrm{EFT}}

\newcommand\pt{\mathrm{pt}}
\newcommand\spt{\mathrm{spt}}

\newcommand\Bord{\mathhyphen\mathrm{Bord}}
\newcommand\EBord{\mathhyphen\mathrm{EBord}}

\newcommand\Vect{\mathrm{Vect}}
\newcommand\SM{\mathrm{SM}}

\newcommand\mathhyphen{\textrm{-}}
\newcommand\mathemdash{\textrm{---}}

\newcommand\sslash{/\kern-0.7ex/}
\newcommand\backsslash{\backslash\kern-0.7ex\backslash}

\begin{document}

\begin{abstract}
  We propose a dimensional reduction procedure for
  $1\vert1$-dimensional supersymmetric Euclidean field theories (EFTs)
  in the sense of Stolz and Teichner.  Our construction is well-suited
  in the presence of a finite gauge group or, more generally, for
  field theories over an orbifold.  As an illustration, we give a
  geometric interpretation of the Chern character for manifolds with
  an action by a finite group.
\end{abstract}

\maketitle
 
\section{Introduction}
\label{sec:introduction}

In the context of topological quantum field theory (that is, the study
of symmetric monoidal functors $d\Bord \to \Vect$ and variants
thereof), dimensional reduction is the assignment of a
$(d-1)$-dimensional theory to a $d$-dimensional theory induced by the
functor of bordism categories $S^1\times\mathemdash \colon (d-1)\Bord
\to d\Bord$.  In the Stolz--Teichner framework of supersymmetric
Euclidean field theories (EFTs) \cites{MR2742432} {MR2079378},
dimensional reduction is a more subtle subject, but it can still be
implemented and provides geometric interpretations of classical
constructions in algebraic topology.  To give the basic idea, we first
recall that $0\vert1$-dimensional EFTs over a manifold $X$ are in
bijection, after passing to concordance classes, with de Rham
cohomology classes of $X$ \cite{MR2763085}.  On the other hand, super
parallel transport \cite{MR2407109} allows us to associate a field
theory $E_V \in 1\vert1\EFT(X)$ to any vector bundle with connection
$V \in \Vect^\nabla(X)$, and a similar statement relating
$1\vert1$-dimensional EFTs and topological $K$-theory is widely
expected.  Moreover, there is a dimensional reduction map
$\mathrm{red}$ between (groupoids of) field theories over $X$ that
recovers the Chern character, in the sense that the diagram
\begin{equation*}
  \xymatrix@R=0pt{
    & 1\vert1\EFT(X)\ar[r]^-{\operatorname{red}}\ar@{.>}[dd]
    & 0\vert1\EFT(X) \ar[dd] \\
    \mathrm{Vect}^{\nabla}(X) \ar@/^2ex/[ur]^E\ar@/_2ex/[dr] &&\\
    & K^0(X) \ar[r]^-{\ch} & H^{\mathrm{ev}}(X; \mathbb C)}
\end{equation*}
commutes \cites{arXiv:0711.3862} {arXiv:1202.2719}.

This paper is part of an ongoing project aiming to identify gauged
supersymmetric field theories as geometric cocycles for equivariant
cohomology theories \cites{StolzGaugedEFTs} {arXiv:1410.5500}
{arXiv:1610.02362}.  Our main goal here is to extend the above
dimensional reduction procedure for $1\vert1$-EFTs to the case where
the manifold $X$ is replaced by an orbifold $\mathfrak X$ (or, more
generally, any stack on the site $\SM$ of supermanifolds).  This will
be based on a series of functors between variants of the Euclidean
bordism categories over $\mathfrak X$,
\begin{equation}
  \label{eq:3}
  0\vert1\EBord(\Lambda\mathfrak X)
  \overset{\mathcal P}\leftarrow 0\vert1\EBord^{\mathbb T}(\Lambda\mathfrak X)
  \overset{\mathcal Q}\to 0\vert1\EBord^{\mathbb R\sslash\mathbb Z}(\Lambda\mathfrak X)
  \overset{\mathcal R}\to 1\vert1\EBord(\mathfrak X).
\end{equation}
Dimensional reduction of field theories (or twist functors) will then
be realized as the pull-push operation induced by $\mathcal R$,
$\mathcal Q$ and $\mathcal P$.  The two middle objects in
\eqref{eq:3}, which we call $\mathbb T$- respectively $\mathbb
R\sslash\mathbb Z$-equivariant bordisms over the inertia stack
$\Lambda\mathfrak X$, as well as the maps involving them, are
introduced in section~\ref{sec:dimensional-reduction}.  Here, $\mathbb
T = \mathbb R/\mathbb Z$ stands for the circle group and $\mathbb
R\sslash\mathbb Z$ is the stack arising from the action of $\mathbb Z$
on $\mathbb R$.  These are of course equivalent as group stacks, and
our terminology just intends to indicate which model for the circle is
directly involved in the definition of each bordism category.  (The
two equivariant bordism categories also turn out to be equivalent,
though this requires proof; see theorem~\ref{thm:Q-is-an-equiv}.)  We
also remark that $\mathcal R$ takes values in the substack of closed
bordisms.  This allows us to avoid delving into the somewhat long
definition of the full bordism category $1\vert1\EBord(\mathfrak X)$,
and focus on the stack $\mathfrak K(\mathfrak X)$ of closed, connected
bordisms, which we call Euclidean supercircles.

As a simple but illustrative application, we specialize to the case
where $\mathfrak X = X\sslash G$ is a global quotient orbifold and
give a field-theoretic interpretation of the simplest instance of
orbifold Chern character, namely the one concerning untwisted
cohomology of global quotients \cite{MR928402}.  It is possible to
extend the map $E$ above for an orbifold $\mathfrak X$ in place of the
manifold $X$; since the dimensional reduction only depends on the
values of a field theory on closed bordisms, we will only describe the
partition function $Z_V$ of the field theory $E_V$ in this paper.
That is, we construct a map
\begin{equation*}
  Z\colon \mathrm{Vect}^\nabla(\mathfrak X)
  \to C^\infty(\mathfrak K(\mathfrak X))
\end{equation*}
(cf.\ section~\ref{sec:parall-transp-field}).  From the discussion of
section~\ref{sec:bordism-categories} it will follow that
$0\vert1$-EFTs over the inertia orbifold $\Lambda\mathfrak X$ are
geometric cocycles for the so-called delocalized cohomology
$H^{\mathrm{ev}}_G(\hat X; \mathbb C)$---the codomain of the
equivariant Chern character $\mathrm{ch}_G$ (cf.\
section~\ref{sec:baum-connes-chern}).  Finally, in
section~\ref{sec:proof-global-quot-ch} we verify that the dimensional
reduction of $Z_V$ is a representative of $\ch_G(V)$.

\begin{theorem}
  \label{thm:global-quot-ch}
  Let $\mathfrak X = X\sslash G$ be the quotient stack arising from
  the action of a finite group on a manifold.  Then the diagram
  \begin{equation*}
    \xymatrix@R=0pt{
      & C^\infty(\mathfrak K(\mathfrak X))\ar[r]^-{\operatorname{red}}
      & 0\vert1\EFT(\Lambda \mathfrak X) \ar[dd] \\
      \mathrm{Vect}^{\nabla}(\mathfrak X) \ar@/^2ex/[ur]^Z\ar@/_2ex/[dr] &&\\
      & K^0_G(X) \ar[r]^-{\ch_G} & H^{\mathrm{ev}}_G(\hat X; \mathbb C)}
  \end{equation*}
  commutes, and moreover the vertical map induces a bijection after
  passing to concordance classes of field theories.
\end{theorem}

\begin{remark}
  In a subsequent paper \cite{arXiv:1801.03016}, we construct twists
  for $1\vert1$-EFTs over $\mathfrak X$ associated to classes in
  $H^3(\mathfrak X, \mathbb Z)$, using a representing gerbe with
  connection as input data, as well as twisted field theories from
  twisted vector bundles.  We also employ the dimensional reduction
  procedure given here to relate these twists and twisted field
  theories with more general versions of the orbifold Chern character.
  In particular, when the twist is trivial, the field theories in
  question do indeed have $Z$ as partition function.  This allows us
  to replace $C^\infty(\mathfrak K(\mathfrak X))$ with a suitably
  defined groupoid $1\vert1\EFT(\mathfrak X)$ of field theories over
  $\mathfrak X$ in the above theorem.
\end{remark}

While this work was in preparation, closely related preprints by
Daniel Berwick-Evans have appeared \cites{arXiv:1410.5500}
{arXiv:1510.07893}.  His approach is heavily inspired by ideas from
perturbative quantum field theory, while ours is more geometric,
putting group actions on stacks at the forefront.

\subsection{Terminology and background}
\label{sec:term-backgr}

For an extensive survey of the Stolz--Teichner program, see
\cite{MR2742432}.  The facts more directly relevant to this paper,
regarding $0\vert1$-dimensional field theories, can be found in
\cite{MR2763085}.  Concerning supermanifolds, we generally follow the
definitions and conventions of \textcite{MR1701597}, and in particular
we routinely use the functor of points formalism.  The necessary facts
about Euclidean structures are reviewed in
appendix~\ref{sec:eucl-superm}.

Vector bundles are always $\mathbb Z/2$-graded and over $\mathbb C$,
and $\Vect^\nabla$ denotes the stack of vector bundles with
connection. $C^\infty$ and $\Omega^*$ denote the sheaves of
complex-valued functions and differential forms.  In the category of
supermanifolds, the notions of principal bundles and connections mimic
the usual definitions \cite{MR1606851}.  If $G$ is a super Lie group
with super Lie algebra $\mathfrak g$, a principal $G$-bundle over the
base $X$ is a manifold $P$ with a free $G$-action and an invariant
submersion $P \to X$ which is locally isomorphic to $X \times G \to
X$.  A connection is a real form $\omega \in \Omega^1(P; \mathfrak g)$
of \emph{even} parity satisfying the usual conditions (to be
$G$-invariant and coincide with the Maurer-Cartan form of $G$ on the
fibers), and its curvature is $d\omega + 1/2[\omega,\omega] \in
\Omega^2(X, P \times_{\mathrm{ad}} \mathfrak g)$.  More generally, if
$X \to S$ is a submersion, then an $S$-family of differential forms,
or fiberwise form, is a section of some exterior power of
$\operatorname{Coker} (T^*S \to T^*X)$.  Fiberwise connections and
their curvature are families of forms defined in a similar fashion.

We treat stacks on the site $\SM$ of supermanifolds (where a covering
is a collection of jointly surjective local diffeomorphisms) in a
geometric way, meaning, for instance, that most of our diagrams
involving manifolds must be interpreted as diagrams in stacks, where
some of the objects happen to be representable sheaves.  A
differentiable stack is a stack $\mathfrak X$ that admits an atlas
$X_0 \to \mathfrak X$, or, equivalently, can be presented by a Lie
groupoid $X_1 \rightrightarrows X_0$.  We recommend the appendix of
\textcite{MR2763085} for a short introduction to stacks, and
\textcite{MR2817778} as a more detailed reference, including the
stacky perspective on orbifolds and cohomology of orbifolds.  The less
standard piece of descent theory needed in this paper concerns group
actions on stacks.  We offer a short overview (with further
references) in appendix~\ref{sec:group-actions-stacks}, where we also
record a lemma that may be of independent interest
(proposition~\ref{prop:6}).

If $\mathfrak X$ is a stack, we define its \emph{inertia} to be the
mapping stack
\begin{equation*}
  \Lambda\mathfrak X
  = \underline\Fun_{\mathrm{SM}}(\mathrm{pt}\sslash \mathbb Z, \mathfrak X).
\end{equation*}
More concretely, $\Lambda\mathfrak X$ is the fibered category whose
$S$-points are given by pairs $(x,\alpha)$ with $x\in \mathfrak X_S$
and $\alpha$ an automorphism of $x$.  A morphism $(x,\alpha) \to (x',
\alpha')$ is given by a morphism $\psi\colon x \to x'$ in $\mathfrak
X$ such that $\alpha'\circ \psi = \psi \circ \alpha$.  The stack
$\mathrm{pt}\sslash \mathbb Z$ can be thought of as a categorical
circle, and $\Lambda\mathfrak X$ is the stack of ``hidden loops'',
i.e., those loops that are not seen by the coarse moduli space of
$\mathfrak X$ (see e.g.\ \cite{MR2068317} for more information).
Notice that $\mathrm{pt}\sslash\mathbb Z$ is a group object in stacks,
and it follows that $\Lambda\mathfrak X$ is acted upon by it.
Concretely, such an action translates as an automorphism of
$\mathrm{id}_{\Lambda\mathfrak X}$, namely the natural transformation
assigning to $(x,\alpha)$ the automorphism $\alpha$.

\subsection{Acknowledgments}
\label{sec:acknowledgments}

This paper is based on a part of my Ph.D. thesis \cite{MR3553617}, and
I would like to thank my advisor, Stephan Stolz, for the guidance.  I
would also like to thank Bertram Arnold, Peter Teichner and Peter
Ulrickson for valuable discussions, Karsten Grove for the financial
support during my last semester as a graduate student (NSF grant
DMS-1209387), and the referee for many careful suggestions.

\section{Bordisms and field theories over an orbifold}
\label{sec:bordism-categories}

A $d$-dimensional topological (quantum) field theory, in the usual
definition of Atiyah and Segal, is a symmetric monoidal functor
\begin{equation*}
  E \in \Fun^\otimes(d\Bord, \Vect)
\end{equation*}
between the category of $d$-dimensional bordisms and the category of
vector spaces.  The domain has as objects closed $(d-1)$-dimensional
manifolds and as morphisms diffeomorphism classes of bordisms between
them.

\Textcite{MR2742432} consider a refinement of the above, where each
bordism is equipped with several additional geometric structures:
supersymmetry, meaning that a bordism is now a supermanifold of
dimension $d\vert\delta$; a Euclidean structure in the sense of
appendix~\ref{sec:eucl-superm}; and finally a smooth map to a fixed
manifold $X$.  In order to make sense of the idea that field theories
should depend smoothly on the input data, we are led to formulate the
resulting bordism category $d\vert\delta\EBord(X)$ as a (weak)
category internal to symmetric monoidal stacks.  This also allows us
to keep track of isometries between bordisms instead of just
considering equivalence classes of bordisms modulo isometry.

Once this framework is in place, it is clear how to replace the
manifold $X$ by a ``generalized manifold'', or stack, $\mathfrak X$:
an $S$-family of bordisms in $d\vert\delta\EBord(\mathfrak X)$ is
given by a submersion $\Sigma \to S$ of codimension $d\vert\delta$
with fiberwise Euclidean structure, an object of $\mathfrak X_\Sigma$
(which, by the Yoneda lemma, corresponds to a map $\psi\colon \Sigma
\to \mathfrak X$ in the realm of generalized manifolds), and lastly
some boundary information we will not detail here (cf.\
\cite[section~2.8]{arXiv:1801.03016}).  A morphism over $f\colon S' \to S$ in the
stack of bordisms is determined by a fiberwise isometry $F\colon
\Sigma' \to \Sigma$ covering $f$ (and suitably compatible with the
boundary information) together with a morphism $\xi$ between objects
of $\mathfrak X_{\Sigma'}$ as indicated in the diagram below.
\begin{equation*}
  \begin{gathered}
    \xymatrix@R-1.5em{ \Sigma'
      \ar[dd]\ar[rd]^F
      \ar@/^3ex/[rrd]^{\psi'}_{}="a"&&\\
      & \Sigma \ar[dd]\ar[r]^{\psi}  \ar_-\xi@{<=}"a"& \mathfrak X\\
      S' \ar[rd]^f &&\\
      & S &}
  \end{gathered}
\end{equation*}

Finally, Euclidean field theories of dimension $d\vert\delta$ over $\mathfrak
X$ are functors of internal categories:
\begin{equation*}
  d\vert\delta\EFT(\mathfrak X) =
  \Fun_{\mathrm{SM}}^\otimes(d\vert\delta\EBord(\mathfrak X), \mathrm{TV}),
\end{equation*}
where $\mathrm{TV}$ is an internal version of the category of
topological vector spaces.  These are contravariant objects on the
variable $\mathfrak X$, and we call two EFTs $E_0, E_1 \in
d\vert\delta\EFT(\mathfrak X)$ concordant if there exist a field
theory $E \in d\vert\delta\EFT(\mathfrak X \times \mathbb R)$ such
that $E \cong \pr_1^*E_0$ on $\mathfrak X \times (-\infty, 0)$ and $E
\cong \pr_1^*E_1$ on $\mathfrak X \times (1, \infty)$.

These observations are the foundation of an equivariant extension of
Stolz--Teichner program.  In this paper, we are only interested in the
cases $d\vert\delta = 0\vert1$ or $1\vert1$, so we can work with
simplified definitions, which we discuss in the remainder of this
section.

\subsection{Dimension $0\vert1$}
\label{sec:dimension-0vert1}

Since every $0\vert1$-dimensional bordism is closed, a $0\vert1$-EFT
is nothing but the assignment of a complex number to each Euclidean
$0\vert1$-manifold, in a way that is invariant under isometries and
such that disjoint unions map to products.  In particular, a
$0\vert1$-EFT is determined by its values on connected bordisms.
Thus, we can just define
\begin{equation*}
  0\vert1\EFT(\mathfrak X) = \Fun_\SM(\mathfrak B(\mathfrak X),
  \mathbb C) = C^\infty(\mathfrak B(\mathfrak X)),
\end{equation*}
where $\mathfrak B(\mathfrak X)$ is some model for the full substack
comprising fiberwise connected bordisms in $0\vert1\EBord(\mathfrak
X)$.  Concretely, we take it to be
\begin{equation*}
  \mathfrak B(\mathfrak X) = \Pi T\mathfrak X \sslash \Isom(\mathbb
  R^{0\vert1}),
  \text{ where }
  \Pi T\mathfrak X = \underline\Fun_\SM(\mathbb R^{0\vert1}, \mathfrak
  X).
\end{equation*}
Here, $\Fun_\SM$ denotes the groupoid of fibered functors and natural
transformations over $\SM$, while $\underline\Fun_\SM$ denotes the
mapping stack.  Thus, $\mathfrak B(\mathfrak X)$ is the quotient stack
arising from a group action on a stack; see
appendix~\ref{sec:group-actions-stacks}.  The notation $\Pi T\mathfrak
X$ is motivated by the fact that when $\mathfrak X$ is a manifold, the
internal hom in question is in fact representable by the
parity-reversed tangent bundle.

\begin{theorem}
  \label{thm:0|1-EFT-orbifold}
  For any differentiable stack $\mathfrak X$, there is a natural
  bijection
  \begin{equation*}
    0\vert1\EFT(\mathfrak X) \cong \Omega^{\mathrm{ev}}_{\mathrm{cl}}(\mathfrak X)
  \end{equation*}
  between $0\vert1$-EFTs over $\mathfrak X$ and closed differential
  forms of even parity.  If $\mathfrak X$ is an orbifold, passing to
  concordance classes gives an isomorphism with even de Rham
  cohomology:
  \begin{equation*}
    0\vert1\EFT(\mathfrak X)/\mathrm{concordance}
    \cong H^\mathrm{ev}_{\mathrm{dR}}(\mathfrak X).  
  \end{equation*}
\end{theorem}

Here, $\Omega^{\mathrm{ev}}_{\mathrm{cl}}(\mathfrak X) =
\Fun_\SM(\mathfrak X, \Omega^{\mathrm{ev}}_{\mathrm{cl}})$, where
$\Omega^{\mathrm{ev}}_{\mathrm{cl}}$ is the sheaf on $\SM$ of even,
closed differential forms.

\begin{proof}
  For $\mathfrak X$ a manifold, this is theorem~1 in
  \textcite{MR2763085}, and the main ingredient of the proof is to
  identify the action of $\Isom(\mathbb R^{0\vert1}) = \mathbb
  R^{0\vert1} \rtimes \mathbb Z/2$ on $\Pi TX$.  It turns out that on
  $C^\infty(\Pi TX) = \Omega^*(X)$, $\mathbb Z/2$ acts as the mod $2$
  grading involution, and the odd vector field generating the $\mathbb
  R^{0\vert1}$-action is precisely the de Rham differential.

  Now, let $X_1 \rightrightarrows X_0$ be a groupoid presentation of
  $\mathfrak X$.  Then $\Pi TX_1 \rightrightarrows \Pi TX_0$ is a
  groupoid presentation of $\Pi T\mathfrak X$, since both stacks
  assign a groupoid equivalent to
  \begin{equation*}
    \SM(S \times \mathbb R^{0\vert1}, X_1) \rightrightarrows \SM(S
    \times \mathbb R^{0\vert1}, X_0)
  \end{equation*}
  to any contractible $S$.  It follows
  \cite[proposition~7.13]{MR2763085} that
  \begin{equation*}
    C^\infty(\Pi T\mathfrak X)
    \cong \lim(\Omega^*(X_0) \rightrightarrows \Omega^*(X_1))
    = \Omega^*(\mathfrak X).
  \end{equation*}
  The $\Isom(\mathbb R^{0\vert1})$-action on $\Omega^*(\mathfrak X)
  \subset \Omega^*(X_0)$ is, again, generated by the de Rham
  differential and the $\mathbb Z/2$-grading operator.

  Now, $0\vert1\EFT(\mathfrak X) = \Fun_{\mathrm{SM}}(\Pi T\mathfrak
  X\sslash \Isom(\mathbb R^{0\vert1}), \mathbb C)$ can be calculated
  from proposition~\ref{prop:6}, and is given by
  \begin{equation*}
    \lim(C^\infty(\Pi T\mathfrak X) \rightrightarrows C^\infty(\Pi
    T\mathfrak X \times \Isom(\mathbb R^{0\vert1}))) =
    \Omega^*(\mathfrak X)^{\Isom(\mathbb R^{0\vert1})} =
    \Omega^{\mathrm{ev}}_{\mathrm{cl}}(\mathfrak X).
  \end{equation*}
  By the Stokes theorem, forms in
  $\Omega^{\mathrm{ev}}_{\mathrm{cl}}(X_0)$ are concordant if and only
  if they are cohomologous.  The same type of argument shows that
  concordance through closed, $X_1$-invariant forms is the same
  relation as being cohomologous in the chain complex
  $\Omega^*(\mathfrak X)$.  Thus
  \begin{equation*}
    0\vert1\EFT(\mathfrak X)/\mathrm{concordance} \cong
    H^{\mathrm{ev}}(\Omega^*(\mathfrak X), d)
  \end{equation*}
  For orbifolds, the right-hand side can be taken as the definition of
  de Rham cohomology \cite[corollary~25]{MR2172499}.
\end{proof}

\begin{remark}
  Differential forms and cohomology classes of odd degree are
  similarly related to field theories twisted by the basic twist
  $\mathcal T_1$ of \cite[definition 6.2]{MR2763085}.  This statement
  is proven similarly, using, as in \cite{MR2763085}, the fact that
  sections of $\mathcal T_1$ correspond to closed, odd differential
  forms (see also \cite[section~2]{arXiv:1801.03016}, where more
  general twists are considered).
\end{remark}

\subsection{Dimension $1\vert1$}
\label{sec:dimension-1vert1}

In order to construct the functor of internal categories $\mathcal R$
of diagram \eqref{eq:3}, all details about the stack of objects of
$1\vert1\EBord(\mathfrak X)$ and non-closed bordisms are entirely
irrelevant; this is, again, due to the fact that the domain of
$\mathcal R$ has trivial object stack.  Thus, it suffices to work with
the moduli stack of closed and connected bordisms in
$1\vert1\EBord(\mathfrak X)$, which we will also call the stack of
\emph{Euclidean supercircles over $\mathfrak X$} and denote by
$\mathfrak K(\mathfrak X)$.

The moduli stack $\mathfrak K(\mathfrak X)$ of Euclidean supercircles
over $\mathfrak X$ is defined as follows.  An object $(K, \psi)$ of
$\mathfrak K(\mathfrak X)$ over $S$ is given by an $S$-family $K$ of
closed, connected Euclidean $1\vert1$-manifolds together with a map
$\psi\colon K \to \mathfrak X$.  A morphism $(K', \psi') \to (K,
\psi)$ over a map $f\colon S' \to S$ is given by a fiberwise isometry
$F\colon K' \to K$ covering $f$ together with a $2$-morphism $\psi'
\rightarrow \psi \circ F$; compositions are performed in the obvious
way.  The data of a morphism can be summarized by the following
diagram.
\begin{equation*}
  \begin{gathered}
    \xymatrix@R-1.5em{
      K' \ar[dd]\ar[rd]^F
      \ar@/^3ex/[rrd]^{\psi'}_{}="a"&&\\
      & K \ar[dd]\ar[r]^{\psi}  \ar_-\xi@{<=}"a"& \mathfrak X\\
      S' \ar[rd]^f &&\\
      & S &}
  \end{gathered}
\end{equation*}

\begin{remark}
  A complete definition of the bordism category
  $1\vert1\EBord(\mathfrak X)$ is given in \cite{arXiv:1801.03016}.
  It is easy to see that $\mathfrak K(\mathfrak X)$, as given here, is
  indeed the substack of closed and connected bordisms there.
  Alternatively, the reader may prefer to think of the present
  description of $\mathfrak K(\mathfrak X)$ as being sufficiently
  reasonable, and thus a sanity check for the more general
  construction.
\end{remark}

A detailed study of the stack $\mathfrak K = \mathfrak K(\mathrm{pt})$
is given in section~\ref{sec:supercircles}.  Examples of (families of)
supercircles can be obtained by choosing a ``length'' parameter
$l\colon S \to \mathbb R^{1\vert1}_{>0}$, and then letting
\begin{equation*}
  K_l = (S \times \mathbb R^{1\vert1})/ \mathbb Z l
\end{equation*}
be given by the orbit space of the translation by $l$.
Proposition~\ref{prop:3} shows that, at least locally in $S$, every
supercircle is of this form (but not canonically).  Moreover, any
morphism $K_{l'} \to K_l$ is determined by a smooth map $S' \to
\mathbb R^{1\vert1} \rtimes \mathbb Z/2$, which fixes a certain
relation between $l'$ and $l$; see section~\ref{sec:supercircles} for
more details.

In order to define the functor $\mathcal R$ in
section~\ref{sec:r-sslash-z-equiv-bord}, we will use an alternative
method to construct supercircles, provided by the following theorem.

\begin{theorem}
  \label{thm:1}
  Let $\Sigma \to S$ be an $S$-family of Euclidean $0\vert
  1$-manifolds and $P \to \Sigma$ a principal $\mathbb T$-bundle.
  Then a fiberwise (in $S$) connection form $\omega$ on $P$ whose
  curvature agrees with the tautological $2$-form $\zeta$ on $\Sigma$
  canonically determines a Euclidean structure on $P$.  Isometries of
  $P$ correspond bijectively to connection-preserving bundle maps
  covering an isometry of $\Sigma$.
\end{theorem}

This is just a restatement of theorem~\ref{thm:K1-versus-BT}, proven
at the end of the paper.

\begin{remark}
  \label{rem:1}
  To see why the data of $\omega$ is essential here, notice that the
  short exact sequence of super Lie groups
  \begin{equation*}
    1 \to \mathbb R \to \mathbb R^{1\vert1} \to \mathbb R^{0\vert1} \to 1
  \end{equation*}
  is not split.  As a consequence, the cartesian product of Euclidean
  manifolds of dimensions $1$ and $0\vert 1$ is not endowed with a
  canonical Euclidean structure.  This makes dimensional reduction in
  our setting quite subtle, since ``crossing with $S^1$'' is not a
  well-defined operation in the Euclidean category, and therefore
  there is no direct functor $S^1 \times \textrm{---}\colon
  0\vert1\EBord \to 1\vert1\EBord$.
\end{remark}

Finally, we remark that every $1\vert1$-EFT over
$\mathfrak X$ determines a smooth function on $\mathfrak K(\mathfrak
X)$, the \emph{partition function} of the theory.  Again, this is an
immediate consequence of the fact that the empty manifold, being the
monoidal unit in the bordism category, is required to map to the
vector space $\mathbb C$.

\section{Dimensional reduction}
\label{sec:dimensional-reduction}

The upshot of section~\ref{sec:bordism-categories} is that it suffices
to discuss the functors \eqref{eq:3} of internal categories in terms
of the corresponding substacks of (fiberwise) closed and connected
families of bordisms; we reserve the term \emph{moduli stack} for
these objects.  We have already discussed $\mathfrak B(\mathfrak X)$
and $\mathfrak K(\mathfrak X)$ in
section~\ref{sec:bordism-categories}.  The two middle moduli stacks,
as well as the maps
\begin{equation*}
  \mathfrak B(\Lambda\mathfrak X) 
  \xleftarrow{\mathcal P}
    \mathfrak B^{\mathbb T}(\Lambda\mathfrak X) 
  \xrightarrow{\mathcal Q}
     \mathfrak B^{\mathbb R\sslash\mathbb Z}(\Lambda\mathfrak X)
  \xrightarrow{\mathcal R} \mathfrak K(\mathfrak X)
\end{equation*}
relating them, will be defined in the ensuing subsections.  We will
refer to the two middle stacks as the $\mathbb T$-equivariant and
$\mathbb R\sslash\mathbb Z$-equivariant moduli stacks of Euclidean
$0\vert1$-manifolds over $\Lambda\mathfrak X$.

The lack of a direct map from left to right in the above span of
moduli stacks is due to a subtlety of super Euclidean geometry: if
$\Sigma$ is a Euclidean $0\vert1$-manifold, the product $S^1 \times
\Sigma$ does not come with a canonical Euclidean structure; to choose
one essentially amounts to the choice of preimage along $\mathcal P$
(cf.\ theorem~\ref{thm:1} and remark~\ref{rem:1}).  This is not a
serious issue for us, since $\mathcal P$ induces a bijection between
the set of functions on each moduli stack (or, equivalently, between
field theories based on each variant of the bordism category; see
proposition~\ref{prop:5}).

Following the physical (and, by now, mathematical) jargon, restriction
of $1\vert1$-EFTs (or just functions on $\mathfrak K(\mathfrak X)$) to
$0\vert1$-EFTs via the above maps of bordism stacks will be referred
to as dimensional reduction.  Our motivation for doing this is that
the stack $\mathfrak K(\mathfrak X)$ of Euclidean supercircles over
$\mathfrak X$ is ``infinite dimensional'', and therefore unwieldy to
analysis; dimensional reduction allows us to probe its geometry by
means of $0\vert1$-dimensional gadgets over $\mathfrak X$.

To further motivate our dimensional reduction procedure, note that
$\mathcal Q$ is an equivalence of stacks
(theorem~\ref{thm:Q-is-an-equiv}), even though its inverse does not
admit a nice geometric description.  Thus, $\mathfrak B^{\mathbb
  T}(\Lambda\mathfrak X)$ and $\mathfrak B^{\mathbb R\sslash\mathbb
  Z}(\Lambda\mathfrak X)$ can be seen as different presentations of
the same entity; the former presentation has a direct relationship
with $\mathfrak B(\Lambda\mathfrak X)$, while the latter leads us to a
suitable definition of a map to $\mathfrak K(\mathfrak X)$.

\begin{remark}
  To understand the relevance of $\mathbb R\sslash\mathbb Z$-actions
  for dimensional reduction, we can consider a naive replacement for
  the composition $\mathcal R \circ \mathcal Q$: instead of performing
  the descent constructions of section~\ref{sec:map-equiv-bordisms},
  we could simply perform a pullback along $P \to \Sigma$.  Then it is
  easy to see that, with these modifications,
  theorem~\ref{thm:global-quot-ch} would recover the naive Chern
  character
  \begin{equation*}
    K_G^0(X)
    \xrightarrow{\alpha} K^0(EG \times_G X)
    \xrightarrow{\mathrm{ch}} H^{\mathrm{ev}}(EG \times_G X)
    = H_G^{\mathrm{ev}}(X)
  \end{equation*}
  (or, more precisely, its pullback to $\hat X$, as defined in
  \eqref{eq:10}).  Here, the map $\alpha$ is given, at the level of
  vector bundles, by the homotopy quotient construction.  Thus, as
  explained at the end of section~\ref{sec:baum-connes-chern}, this
  alternative construction forgets too much information.
\end{remark}

At the end of this section, to illustrate the ideas, we specialize
these constructions to the case where $\mathfrak X = X\sslash G$ is a
global quotient by a finite group.

\subsection{The $\mathbb R\sslash\mathbb Z$-equivariant moduli stack
  and the map $\mathcal R$}
\label{sec:r-sslash-z-equiv-bord}

We define a stack $\mathfrak B^{\mathbb R\sslash\mathbb
  Z}(\Lambda\mathfrak X)$ where an object over $S$ is given by the
following data:
\begin{enumerate}
\item a family $\Sigma \to S$ of connected Euclidean $0|1$-manifolds,
\item a principal $\mathbb T$-bundle $P \to \Sigma$ with a fiberwise
  connection $\omega$ whose curvature agrees with the tautological
  (fiberwise) $2$-form $\zeta$ on $\Sigma$ (see
  appendix~\ref{sec:eucl-superm}), and
\item an $\mathbb R\sslash\mathbb Z$-equivariant map $\psi \colon P
  \to \Lambda\mathfrak X$ with equivariance datum $\rho$, where
  $\mathbb R\sslash\mathbb Z$ acts on $P$ and $\Lambda\mathfrak X$ via
  the usual homomorphisms $\mathbb R\sslash\mathbb Z \to \mathbb T$
  respectively $\mathbb R\sslash\mathbb Z \to
  \mathrm{pt}\sslash\mathbb Z$.
\end{enumerate}
(Recall that $\mathbb R\sslash\mathbb Z$-equivariance is not just a
condition on $\psi$, but rather extra data encoded by the $2$-morphism
$\rho$, see appendix~\ref{sec:group-actions-stacks}).  We will usually
denote this object $(\Sigma, P, \psi, \rho)$ or, diagrammatically,
\begin{equation*}
  \xymatrix{P \ar[r]^\psi_{\mathbb R\sslash\mathbb Z} \ar[d] & \Lambda\mathfrak X \\ \Sigma.}
\end{equation*}
A morphism $(\Sigma', P', \psi', \rho') \to (\Sigma, P, \psi, \rho)$
covering a map of supermanifolds $S' \to S$ is given by
\begin{enumerate}
\item a fiberwise isometry $F \colon\Sigma' \to \Sigma$ covering
  $S' \to S$,
\item a connection-preserving bundle map $\Phi\colon P' \to P$
  covering $F$, and
\item an equivariant $2$-morphism $\xi\colon \psi' \to \psi \circ
  \Phi$.
\end{enumerate}
Compositions are performed as suggested by the geometry.

Now we discuss the map $\mathcal R\colon \mathfrak B^{\mathbb R\sslash
  \mathbb Z}(\Lambda\mathfrak X) \to \mathfrak K(\mathfrak X)$.  An
object $(\Sigma, P, \psi, \rho)$ over $S$ is mapped to the supercircle
over $\mathfrak X$ consisting of
\begin{enumerate}
\item the family of $1|1$-dimensional manifolds $P$ endowed with the
  fiberwise Euclidean structure determined by $\omega$ (see
  theorem~\ref{thm:1}), and
\item the map $P \to \mathfrak X$ obtained by composing $\psi$ with the
  forgetful map $\Lambda\mathfrak X \to \mathfrak X$.
\end{enumerate}
Notice that this construction forgets the $\mathbb T$-action on $P$ as
well as the equivariance datum $\rho$.  To define $\mathcal R$ at the
level of morphisms, recall, again by theorem~\ref{thm:1}, that a
connection-preserving bundle map $P' \to P$ covering a fiberwise
isometry $\Sigma' \to \Sigma$ is a fiberwise (over $S$) isometry with
respect to the Euclidean structures on $P'$, $P$.

\subsection{The $\mathbb T$-equivariant moduli stack and the map
  $\mathcal P$}
\label{sec:t-equiv-bord}

For any stack $\mathfrak X$, we define $\mathfrak B^{\mathbb
  T}(\mathfrak X)$ to be the stack whose $S$-points are given by an
$S$-family of connected Euclidean $0\vert1$-manifolds $\Sigma \to S$
together with two pieces of data:
\begin{enumerate}
\item a principal $\mathbb T$-bundle $P \to \Sigma$ with a fiberwise
  connection $\omega$ whose curvature agrees with the tautological
  $2$-form $\zeta$ on $\Sigma$,
\item a map $\psi\colon \Sigma \to \mathfrak X$.
\end{enumerate}
Morphisms between two objects $(\Sigma', P', \psi')$ and $(\Sigma, P,
\psi)$ over $f\colon S' \to S$ consist of a fiberwise isometry
$F\colon \Sigma' \to \Sigma$ covering $f$, a connection-preserving
bundle map $\Phi\colon P' \to P$ covering $F$, and a $2$-morphism
$\xi\colon \psi' \to \psi \circ F$.  Compositions are performed as
suggested by the geometry.

The data (1) and (2) above are completely unrelated in the sense that
\begin{equation*}
  \mathfrak B^{\mathbb T}(\mathfrak X) \cong \mathfrak B^{\mathbb T}
  \times_{\mathfrak B} \mathfrak B(\mathfrak X),
\end{equation*}
and our map $\mathcal P\colon \mathfrak B^{\mathbb T}(\mathfrak X) \to
\mathfrak B(\mathfrak X)$ is simply the projection onto the second
component.  Our interest in $\mathfrak B^{\mathbb T}(\mathfrak X)$ is
due to the fact that it admits a straightforward quotient stack
presentation.  Write $\mathbb T^{1\vert1} = \mathbb
R^{1\vert1}/\mathbb Z$ for the (length $1$) super circle group.

\begin{proposition}
  \label{prop:2}
  There is an equivalence of stacks
  \begin{equation*}
    \Pi T \mathfrak X \sslash \Isom(\mathbb T^{1\vert1})
    \to \mathfrak B^{\mathbb T}(\mathfrak X),
  \end{equation*}
  where the action of $\Isom(\mathbb T^{1\vert1})$ on $\Pi T\mathfrak
  X$ is through the quotient
  \begin{equation*}
    \pi\colon \Isom(\mathbb T^{1\vert1}) =
    \mathbb T^{1\vert1} \rtimes \mathbb Z/2 \to \mathbb R^{0\vert1}
    \rtimes \mathbb Z/2 = \Isom(\mathbb R^{0\vert1}).
\end{equation*}
\end{proposition}

\begin{proof}
  For $\mathfrak X = \mathrm{pt}$, this follows from (the proof of)
  theorem~\ref{thm:K1-versus-BT}.  Therefore, in the general case we have
  \begin{equation*}
    \mathfrak B^{\mathbb T}(\mathfrak X) \cong
    \Pi T\mathfrak X\sslash\Isom(\mathbb R^{0\vert1})
    \times_{\mathrm{pt}\sslash\Isom(\mathbb R^{0\vert1})}
    \mathrm{pt}\sslash\Isom(\mathbb T^{1\vert1})
  \end{equation*}
  and the result follows from proposition~\ref{prop:4}.
\end{proof}

\begin{proposition}
  \label{prop:5}
  For any sheaf $\mathfrak F$, the map $\mathcal P\colon \mathfrak
  B^{\mathbb T}(\mathfrak X) \to \mathfrak B(\mathfrak X)$ induces a
  bijection
  \begin{equation*}
    \Fun_{\mathrm{SM}}(\mathfrak B(\mathfrak X), \mathfrak F)
    \to \Fun_{\mathrm{SM}}(\mathfrak B^{\mathbb T}(\mathfrak X),
    \mathfrak F).
  \end{equation*}
\end{proposition}

\begin{proof}
  Under the identification of the previous proposition, $\mathcal
  P\colon \mathfrak B^{\mathbb T}(\mathfrak X) \to \mathfrak
  B(\mathfrak X)$ becomes the natural map
  \begin{equation*}
    \Pi T\mathfrak X\sslash \Isom(\mathbb T^{1\vert1})
    \to \Pi T\mathfrak X\sslash\Isom(\mathbb R^{0\vert1})
  \end{equation*}
  induced by the surjection $\pi\colon \Isom(\mathbb T^{1\vert1}) \to
  \Isom(\mathbb R^{0\vert1})$.  Thus, proposition~\ref{prop:6}
  identifies the set $\Fun_{\mathrm{SM}}(\mathfrak B(\mathfrak X),
  \mathfrak F)$ with the subset of $\Isom(\mathbb
  R^{0\vert1})$-invariants in $\Fun_{\mathrm{SM}}(\Pi T\mathfrak X,
  \mathfrak F)$, and similarly for $\Fun_{\mathrm{SM}}(\mathfrak
  B^{\mathbb T}(\mathfrak X), \mathfrak F)$.  This proves the claim.
\end{proof}

Taking $\mathfrak F = \mathbb C$, we get a bijection between
$0\vert1\EFT(\mathfrak X)$ and $C^\infty(\mathfrak B^{\mathbb
  T}(\mathfrak X))$.  This shows that the last step of our dimensional
reduction procedure, pushing forward along $\mathcal P$, is
well-defined.

\subsection{The map $\mathcal Q\colon
  \mathfrak B^{\mathbb T}(\Lambda\mathfrak X)
  \to \mathfrak B^{\mathbb R\sslash\mathbb Z}(\Lambda\mathfrak X)$}
\label{sec:map-equiv-bordisms}

We denote by $\alpha$ the canonical automorphism of the identity of
$\Lambda\mathfrak X$.  It suffices to describe the restriction of the
desired map $\mathcal Q\colon \mathfrak B^{\mathbb T}(\Lambda\mathfrak
X) \to \mathfrak B^{\mathbb R\sslash\mathbb Z}(\Lambda\mathfrak X)$ to
the full prestack of objects where all bundles involved are trivial.
To $(\Sigma, P, \psi) \in \mathfrak B^{\mathbb T}(\Lambda\mathfrak
X)_S$ with
\begin{equation*}
  \Sigma = S \times \mathbb R^{0\vert1},
  \quad P = S \times \mathbb T^{1\vert1},
  \quad \psi\colon \Sigma \to \Lambda\mathfrak X,
\end{equation*}
and the standard Euclidean structure and connection form, we want to
assign an object $(\Sigma, P, \psi_!\colon P \to \Lambda\mathfrak X,
\rho) \in \mathfrak B^{\mathbb R\sslash\mathbb Z}(\mathfrak X)_S$.
Consider the covering $ U = S \times \mathbb R^{1\vert1} \to P$.  Our
goal is to descend $\tilde \psi\colon U \to \Lambda\mathfrak X$, the
pullback of $\psi$ via $U \to \Sigma$, to a map $\psi_!\colon P \to
\Lambda\mathfrak X$.
\begin{equation*}
  \xymatrix{U \ar[r]\ar@/^{4.2ex}/[rrr]^(0.15){\tilde\psi} &
    P \ar[r]\ar@{-->}@/_{3.4ex}/[rr]_(.2){\psi_!} & \Sigma \ar[r]^(0.3)\psi & \Lambda\mathfrak X}
\end{equation*}
In order to do that, we need to provide certain isomorphisms over
double overlaps and then check a coherence condition on triple
overlaps.  Denote by $\pr_1, \pr_2\colon U \times_P U
\rightrightarrows U$ the projections.  Then we are looking for a
$2$-morphism $\tilde\alpha\colon \tilde \psi \circ \pr_1 \to \tilde
\psi \circ \pr_2$ (whose domain and codomain happen to be the same
map, henceforth denoted $\psi \circ \pr$).  Note that $U \times_P U$
breaks up as a disjoint union indexed by $\mathbb Z$, where the $n$th
component comprises pairs of the form $(x,n\cdot x)$.  On that
component, we set $\tilde\alpha$ to be the horizontal composition
(whiskering)
\begin{equation*}
  \tilde\alpha = \alpha^n \circ (\psi\circ \pr).
\end{equation*}
Regarding the coherence condition, we need to check that
\begin{equation}
  \label{eq:1}
  \pr_{13}^* \tilde\alpha = \pr_{23}^* \tilde\alpha \circ
    \pr_{12}^* \tilde\alpha,
\end{equation}
where $\pr_{ij}$ denotes the projection $U \times_P U \times_P U \to U
\times_p U$ forgetting the third index.  The threefold fiber product
breaks up as a disjoint union indexed by $\mathbb Z \times \mathbb Z$,
where the component $(n,m)$ and its image through the $\pr_{ij}$ are
as follows.
\begin{equation*}
  \xymatrix{&(x, n\cdot x, (n+m)\cdot x) \ar@{|->}[dl]_{\pr_{12}}
    \ar@{|->}[d]^{\pr_{13}}\ar@{|->}[dr]^{\pr_{23}}&\\
  (x,n\cdot x) & (x,(n+m)\cdot x) & (n\cdot x, (n+m)\cdot x)}
\end{equation*}
Therefore, on that component,
\begin{gather*}
    \pr_{23}^* \tilde\alpha = \alpha^m \circ  (\psi\circ\pr), \qquad
    \pr_{12}^* \tilde\alpha = \alpha^n \circ  (\psi\circ\pr),
    \\
    \pr_{13}^* \tilde\alpha = \alpha^{n+m} \circ (\psi\circ\pr),
\end{gather*}
and their vertical compositions are as required by~\eqref{eq:1}.  We
thus obtain the desired $\psi_!\colon P \to \Lambda\mathfrak X$.

Next, we need to provide the $\mathbb R\sslash\mathbb Z$-equivariance
datum $\rho$ for $\psi_!$.  To analyze the putative square
\begin{equation}
  \label{eq:21}
  \begin{gathered}
    \xymatrix{ P \times \mathbb R\sslash\mathbb Z \ar[d]_{\psi_!
        \times \id}
      \ar[r]^-\mu & P \ar[d]^{\psi_!}\\
      \Lambda\mathfrak X \times \mathbb R\sslash\mathbb Z \ar[r]
      \ar@{=>}[ur]_-{\rho}& \Lambda\mathfrak X}
  \end{gathered}
\end{equation}
we notice that, after a suitable base change, any $S$-point of $P
\times \mathbb R\sslash\mathbb Z$ can be pulled back from the atlas
$i_0\colon P \times \mathbb R \to P \times \mathbb R\sslash \mathbb
Z$, or, for that matter, from any of the atlases $i_n\colon
(p,t)\mapsto i_0(p, t+n)$, where $n \in \mathbb Z$; moreover, any
morphism of $S$-points can be pulled back from $m\colon i_n\to
i_{n+m}$.  Thus, we can extract all information encoded by $\rho$ by
evaluating the above diagram on each $i_n$ and $m$.  The top-right
composition factors through $P \times \mathbb T$, so every $i_n$ maps
to the same $\mu^* \psi_! \in \Lambda\mathfrak X_{P \times \mathbb
  R}$, and $m$ maps to the identity.  The left-bottom composition
factors through $\Lambda\mathfrak X \times \mathrm{pt}\sslash\mathbb
Z$, so, for any $n$, $i_n$ maps to $\pr_1^*\psi_! \in \Lambda\mathfrak
X_{P \times \mathbb R}$, and $m\colon i_n \to i_{n+m}$ maps to
$\pr_1^*\alpha^m\colon \pr_1^*\psi_!  \to \pr_1^*\psi_!$.  For each
$i_n$, the fibered natural transformation $\rho$ should give a
morphism $\rho(i_n)\colon \pr_1^*\psi_! \to \mu^*\psi_!$ fitting in
the diagram below.
\begin{equation*}
  \xymatrix@C=10ex{
    \pr_1^*\psi_! \ar[d]_{\pr_1^*\alpha^m}  \ar[r]^{\rho(i_n)}
    &
    \mu^*\psi_! \ar@{=}[d]
    \\
   \pr_1^*\psi_!\ar[r]^{\rho(i_{n+m})}
    &    \mu^*\psi_!
    \\ }
\end{equation*}
This means $\rho$ is completely specified by $\rho(i_0)$, and
naturality imposes no further restrictions on the latter.  To provide
$\rho(i_0)$, it suffices to give a morphism $\pr_1^*\tilde\psi \to
\mu^*\tilde\psi$, where the latter is the composition
\begin{equation*}
  U \times \mathbb R \to P \times \mathbb R \xrightarrow{\mu} P \to
  \Sigma \xrightarrow{\psi} \Lambda\mathfrak X,
\end{equation*}
satisfying appropriate coherence conditions on $U \times_P U \times
\mathbb R$.  Since $\mu^*\tilde\psi = \pr_1\tilde\psi$, we can take
that to be the identity.  One can check that $\rho$ satisfies the
coherence conditions required of the equivariance datum.

The effect of $\mathfrak B^{\mathbb T}(\Lambda\mathfrak X) \to
\mathfrak B^{\mathbb R\sslash\mathbb Z}(\Lambda\mathfrak X)$ on
morphisms is also given by descent.  Given a morphism in $\mathfrak
B^{\mathbb T}(\Lambda\mathfrak X)$
\begin{equation*}
  \xymatrix@R=0.4ex{P \ar[dd]\ar[r]^\Phi & P' \ar[dd] &\\
    && \Lambda\mathfrak X\\
  \Sigma \ar[r]^F \ar@/^{3ex}/[rru]|\hole^(.7)\psi_(.7){}="a" & \Sigma'
  \ar@/_{1.2ex}/[ru]_(.7){\psi'}
 \ar@{<=}"3,2";"a"_-\xi}
\end{equation*}
where $\Sigma' = S' \times \mathbb R^{0\vert1}$, $P' = S' \times
\mathbb T^{1\vert1}$ are also trivial families, consider the fiberwise
universal cover $U' = S\times \mathbb R^{1\vert1} \to P'$ and choose a
lift $\tilde\Phi\colon U \to U'$.  We can then lift $\psi$, $\psi'$
and $\xi$ by composing respectively whiskering with $U\to \Sigma$ or
$U' \to \Sigma'$
\begin{equation*}
  \xymatrix@C=2ex{U \ar[rr]^{\tilde\Phi} \ar[rd]_-{\tilde \psi}^{}="a"&& U'
    \ar[ld]^-{\tilde \psi'}_{}="b"\\ & \Lambda\mathfrak X \ar@{=>}"a";"b"^{\tilde\xi}}
\end{equation*}
and descend $\tilde\xi$ to a morphism $\xi_!\colon \psi_! \to
\Phi^*\psi_!'$.  To justify that, we need to show that on the $n$th
component of $U \times_P U$ the diagram
\begin{equation*}
  \xymatrix{\pr_1^*\tilde \psi\ar[r]^-{\pr_1^*\tilde\xi} \ar[d]_{\alpha^n} &
    \pr_1^*(\tilde \psi' \circ \tilde\Phi)\ar[d]^{\alpha^{n}} \\
    \pr_2^*\tilde \psi \ar[r]^-{\pr_2^*\tilde\xi} & \pr_2^*(\tilde \psi'\circ \tilde\Phi)}
\end{equation*}
commutes.  (To be precise, $\alpha^n$ above stands, respectively, for
$\alpha^n\circ (\psi \circ \pr)$, the gluing isomorphism used to build
$\psi_!$, and its counterpart for $\Phi^*\psi_!'$.)  This follows
immediately from the compatibility condition between $\xi$ and
$\alpha$, namely $\xi \circ \alpha_\psi = \alpha_{\Phi^*\psi'}\circ
\xi$.  The morphism $\xi_!$ thus obtained is independent of the choice
of lift $\tilde\Phi$, since it only depends on the composition $\tilde
\psi'\circ\tilde\Phi$.  We omit the verification that $\xi_!$ is
compatible with the equivariance data.

Finally, we assign to the morphism in $\mathfrak B^{\mathbb
  T}(\Lambda\mathfrak X)$ prescribed by the data $(F, \Phi, \xi)$ the
morphism in $\mathfrak B^{\mathbb R\sslash\mathbb Z}(\Lambda\mathfrak
X)$ prescribed by the data $(F, \Phi, \xi_!)$.  That this assignment
respects compositions follows from uniqueness for descent of
morphisms.

This finishes the construction of $\mathcal Q$.  The next result is
not used directly in the remainder of the paper, and is rather meant
as a motivation for introducing the stacks $\mathfrak B^{\mathbb
  T}(\Lambda \mathfrak X)$ and $\mathfrak B^{\mathbb R\sslash\mathbb
  Z}(\Lambda\mathfrak X)$, which turn out to be just different
presentations of the same object.  In fact, they are presentations
adapted to establishing a relationship with $\mathfrak
B(\Lambda\mathfrak X)$ respectively $\mathfrak K(\mathfrak X)$, as
witnessed by the relatively easy definition of the maps $\mathcal P$
and $\mathcal R$ above.  Note also that the proof of
theorem~\ref{thm:Q-is-an-equiv} is indirect, and does not explicitly
provide an inverse to $\mathcal Q$; thus, it does not seem possible to
simplify our presentation of the dimensional reduction procedure by
removing any mention to $\mathfrak B^{\mathbb T}(\Lambda\mathfrak X)$.

\begin{theorem}
  \label{thm:Q-is-an-equiv}
  The fibered functor $\mathcal Q\colon \mathfrak B^{\mathbb
    T}(\Lambda\mathfrak X) \to \mathfrak B^{\mathbb R\sslash\mathbb
    Z}(\Lambda\mathfrak X)$ is an equivalence.
\end{theorem}

\begin{proof}
  At the morphism level, the effect of the functor in question was
  described in two steps: $\xi \mapsto \tilde\xi \mapsto \xi_{!}$.
  This is a one-to-one procedure because the first step is injective
  (since $U\to\Sigma$ has local sections) and the second step
  (descent) is in fact bijective.  Thus, it remains to show that the
  fibered functor $\mathfrak B^{\mathbb T}(\Lambda\mathfrak X) \to
  \mathfrak B^{\mathbb R\sslash\mathbb Z}(\Lambda\mathfrak X)$ is full
  and essentially surjective.  In order to do that, we will build a
  prestack $\mathfrak B^{\mathrm{triv}}$ and a factorization
  \begin{equation}
    \label{eq:20}
    \begin{gathered}
      \xymatrix@C=1ex{
        &\mathfrak B^{\mathrm{triv}} \ar[dl]_v\ar[dr]^u
        \\
        \mathfrak B^{\mathbb T}(\Lambda\mathfrak X) \ar[rr]^-{\mathcal Q}
        && \mathfrak B^{\mathbb R\sslash\mathbb Z}(\Lambda\mathfrak X).}
    \end{gathered}
  \end{equation}
  where $u$ is full and essentially surjective on the groupoid of
  $S$-point for any contractible $S$.

  The prestack $\mathfrak B^{\mathrm{triv}}$ is defined as follows:
  \begin{enumerate}
  \item an object consists of an object $(\Sigma, P, \psi,\rho) \in
    \mathfrak B^{\mathbb R\sslash\mathbb Z}(\Lambda\mathfrak X)$
    together with a section $s\colon \Sigma \to P$, and
  \item a morphism $(\Sigma', P', \psi', \rho', s') \to (\Sigma, P,
    \psi, \rho, s)$ is a pair consisting of a morphism $(F, \Phi,
    \xi)$ of the underlying objects in $\mathfrak B^{\mathbb
      R\sslash\mathbb Z}(\mathfrak X)$ together with a map $r\colon
    \Sigma' \to \mathbb R$ relating $s$ and $s'$ in the sense that
    $\Phi \circ s' = (s \circ F)e^{2\pi ir}$.
  \end{enumerate}
  With a little poetic license, a morphism can be depicted as follows
  (the square containing $r$ would literally make sense, as a
  $2$-commutative diagram, if we replaced $P$ with $P\sslash\mathbb
  R$).
  \begin{equation}
    \label{eq:5}
    \begin{gathered}
      \xymatrix@R-1.5em{ P'
        \ar[dd]\ar[rd]^\Phi
        \ar@/^3ex/[rrd]^{\psi'}_{}="a"&&\\
        & P \ar[dd]\ar[r]^{\psi}  \ar_-\xi@{<=}"a"& \Lambda\mathfrak X\\
        \Sigma' \ar[rd]^F \ar@{-->}@<1.12ex>[uu]^{s'} \ar@{=>}[ur]^{r}&&\\
        & \Sigma \ar@{-->}@<-1.12ex>[uu]_s&}
    \end{gathered}
  \end{equation}

  We define $u\colon \mathfrak B^{\mathrm{triv}} \to \mathfrak
  B^{\mathbb R\sslash\mathbb Z}(\Lambda\mathfrak X)$ to be the
  forgetful functor, which simply discards $s$ and $r$, so it is
  clearly full and essentially surjective over contractible $S$ as
  claimed.

  Next, we construct $v\colon \mathfrak B^{\mathrm{triv}} \to
  \mathfrak B^{\mathbb T}(\Lambda\mathfrak X)$.  To an object
  $(\Sigma, P, \psi, \rho, s) \in \mathfrak B^{\mathrm{triv}}$, we
  assign the object $(\Sigma, P, s^*\psi)$ in $\mathfrak B^{\mathbb
    T}(\Lambda\mathfrak X)$.  Now fix a morphism as in \eqref{eq:5}.
  To define its image in $\mathfrak B^{\mathbb T}(\Lambda\mathfrak
  X)$, the only new data we need to provide is a morphism $(s')^*\psi'
  \to (s\circ F)^*\psi$, which we take to be the following
  composition:
  \begin{equation*}
    (s')^*\psi' \xrightarrow{(s')^*\xi} (s')^*\Phi^*\psi \cong
                  (\Phi \circ s')^*\psi = ((s\circ F)e^{2\pi ir})^*\psi
                \xrightarrow{\rho_{s\circ F, r}^{-1}}
                  (s\circ F)^*\psi.
  \end{equation*}
  We omit the verification of functoriality.

  To finish the proof, we just need to show that \eqref{eq:20}
  commutes (up to $2$-isomorphism).  It suffices to look at $(\Sigma,
  P, \psi, \rho, s) \in \mathfrak B^{\mathrm{triv}}$ where $P$ and
  $\Sigma$ are trivial families, and pick $s$ to be the unit section;
  our goal is to produce an isomorphism between $(s^*\psi)_!$ and
  $\psi$, natural in the input data $(\Sigma, P, \psi, \rho, s)$ and
  compatible with the respective equivariance data.  From the
  discussion leading to the construction of the $\rho$ in
  \eqref{eq:21}, we see that the data of the present (arbitrarily
  given) $\rho$ is essentially an isomorphism $\rho_0\colon
  \pr_1^*\psi \to \mu^*\psi$ in $\Lambda\mathfrak X_{P \times \mathbb
    R}$.  Now, let $\pi^*\psi$ be the pullback through $\pi\colon U
  \to P$ and recall that $\widetilde{s^*\psi}$ is the $U$-point of
  $\Lambda \mathfrak X$ used to put together $(s^*\psi)_{!}$.  Note
  that each half of the diagram
  \begin{equation*}
    \xymatrix{
      U = \Sigma\times\mathbb R
        \ar[r]^-{s \times \mathrm{id}}
        \ar@/^{6ex}/[rrr]_{\widetilde{s^*\psi}}\ar@/_{6ex}/[rrr]^{\pi^*\psi} &
      P \times\mathbb R
        \ar@<0.66ex>[r]^-{\pr_1}\ar@<-.66ex>[r]_-\mu &
      P
        \ar[r]^(0.3)\psi &
      \Lambda\mathfrak X}
  \end{equation*}
  commutes, so $\rho_0$ gives a morphism $\widetilde{s^*\psi} \to
  \pi^*\psi$ and, by descent, a morphism $(s^*\psi)_{!} \to \psi$.  We
  omit the naturality and compatibility checks.
\end{proof}

\subsection{Global quotients}
\label{sec:global-quotients}

Let us illustrate the above constructions when $\mathfrak X = X\sslash
G$ is the quotient orbifold associated to the action of a finite group
$G$ on a manifold $X$.

We start noticing that a quotient stack presentation for $\Lambda
(X\sslash G)$ can be given as follows.  Consider the product $X \times
G$ with diagonal $G$-action, where $G$ acts on itself by conjugation.
There is an invariant submanifold
\begin{equation}
  \label{eq:10}
  \hat X = \{ (x,g) \in X \times G \mid x \in X^g \},
\end{equation}
and an object over $S$ in the quotient stack $\hat X\sslash G$
consists of a pair $(Q, (f, A))$, where $Q \to S$ is a principal
$G$-bundle and $(f, A)\colon Q \to \hat X \subset X \times G$ is a
$G$-equivariant smooth map.  Denote by $\alpha\colon Q \to Q$ the
bundle automorphism determined by $A$; on $T$-points, it is given by
\begin{equation*}
  \alpha(q) = q A(q), \quad q \in Q_T.
\end{equation*}
Notice that this automorphism preserves $f$, and therefore $(Q, f,
\alpha)$ determines an $S$-point of $\Lambda(X\sslash G)$.
Conversely, given an $S$-point $(Q, f, \alpha)$ of $\Lambda(X\sslash
G)$, we can specify a $G$-equivariant map $A\colon Q \to G$ by
requiring that the above equation holds, and compatibility between $f$
and $\alpha$ implies that the resulting map $(f, A)\colon Q \to X
\times G$ factors through $\hat X$, thus determining an object of
$\hat X\sslash G$ over $S$.

The translation back and forth between $A$ and $\alpha$ provides a
$\pt\sslash\mathbb Z$-equivariant equivalence between
$\Lambda(X\sslash G)$ and $\hat X\sslash G$, compatible with the maps
$\Lambda(X\sslash G) \to X\sslash G$ forgetting the prescribed
automorphism and $\hat X\sslash G \to X\sslash G$ induced the
projection $\pr_1\colon X \times G \to X$.  We will shift freely
between these two formulations.

The geometric content of an $S$-family in $\mathfrak B^{\mathbb
  R\sslash\mathbb Z}(\hat X \sslash G)$ is the following:
\begin{enumerate}
\item a family $\Sigma \to S$ of connected Euclidean
  $0\vert1$-manifolds,
\item a principal $\mathbb T$-bundle $P \to \Sigma$ with a fiberwise
  connection $\omega$ whose curvature agrees with the tautological
  $2$-form on $\Sigma$,
\item a principal $G$-bundle $Q \to P$,
\item\label{item:1} a $G$-equivariant map $(f,A)\colon Q \to \hat X
  \subset X \times G$; or, equivalently, a bundle automorphism
  $\alpha\colon Q \to Q$ and a $G$-equivariant map $f\colon Q \to X$
  such that $f\circ\alpha = f$, and, finally,
\item\label{item:2} a collection of natural isomorphisms of
  $G$-torsors
  \begin{equation*}
    \rho_{p,t}\colon Q_p \to Q_{pe^{2\pi it}}
  \end{equation*}
  for each pair of $T$-points $p\colon T \to P$, $t\colon T \to\mathbb R$,
  intertwining the maps
  \begin{equation*}
    f_p \colon Q_p \to X, \quad
    f_{pe^{2\pi it}} \colon Q_{pe^{2\pi it}} \to X
  \end{equation*}
  and subject to the condition that for any $n\colon T \to \mathbb Z$ the
  diagram
  \begin{equation}
    \label{eq:4}
    \begin{gathered}
      \xymatrix{Q_p \ar[d]_{\alpha_p^n}\ar[r]^-{\rho_{p,t}}
        & Q_{pe^{2\pi it}} \ar@{=}[d]\\
        Q_p \ar[r]^-{\rho_{p, t+n}} & Q_{pe^{2\pi i(t+n)}}}
    \end{gathered}
  \end{equation}
  commutes.
\end{enumerate}
The last condition means that $\alpha$ agrees with the holonomy of $Q$
around the fibers of $P$.  A morphism in $\mathfrak B^{\mathbb
  R\sslash\mathbb Z}(\hat X \sslash G)$ is given by a fiberwise
isometry $F\colon\Sigma' \to \Sigma$, a connection-preserving bundle
map $\Phi\colon P' \to P$ covering $F$, and a bundle map $Q' \to Q$
covering $\Phi$ which is required to be compatible in the obvious way
with the data in \eqref{item:1} and \eqref{item:2} above.

The geometric content of an $S$-family in $\mathfrak B^{\mathbb
  T}(\hat X \sslash G)$ is the following:
\begin{enumerate}
\item a family $\Sigma \to S$ of connected Euclidean $0|1$-manifolds,
\item a principal $\mathbb T$-bundle $P \to \Sigma$ with a connection
  $\omega$ whose curvature agrees with the tautological $2$-form on
  $\Sigma$,
\item a principal $G$-bundle $Q \to \Sigma$, and
\item a $G$-equivariant map $f\colon Q \to \hat X$.
\end{enumerate}
A morphism $(\Sigma', P', Q', f') \to (\Sigma, P, Q, f)$ consists of a
fiberwise isometry $F\colon\Sigma' \to \Sigma$, a
connection-preserving bundle map $\Phi\colon P' \to P$ covering $F$,
and a bundle map $Q' \to Q$ covering $F$ and intertwining the maps
$f\colon Q \to \hat X$ and $f'\colon Q' \to \hat X$.
From proposition~\ref{prop:2}, it follows that $\mathfrak B^{\mathbb
  T}(\hat X \sslash G)$ admits the presentation
\begin{equation*}
  (\Pi T(\hat X\sslash G))\sslash \Isom(\mathbb T^{1\vert1}) \cong \Pi T\hat X\sslash(\Isom(\mathbb T^{1\vert1}) \times G).
\end{equation*}

Finally, let us describe the map $\mathcal Q$ relating the $\mathbb
T$-equivariant and $\mathbb R\sslash\mathbb Z$-equivariant moduli
stacks of Euclidean $0\vert1$-manifolds in this special situation.
Fix $(\Sigma, P, Q, f) \in \mathfrak B^{\mathbb T}(\hat X\sslash G)_S$
and let $(\Sigma, P, Q_!, f_!,\rho) \in \mathfrak B^{\mathbb
  R\sslash\mathbb Z}(\hat X\sslash G)_S$ be its image.  Locally in
$S$, $f$ determines a conjugacy class of $G$ and $Q_! \to P$ is the
$G$-bundle with that holonomy around the fibers of $P \to \Sigma$.
More specifically, let us assume $P$ and $Q$ are trivial; if $S$ is
connected, then $f$ determines an element $g \in G$, namely the one
corresponding to the connected component of $\hat X = \amalg_{g\in
  G}X^g$ in which $f|_{\Sigma \times \{ e \}}$ takes values.  Then
$Q_! \to P$ is the $G$-bundle built as a quotient
\begin{equation*}
  Q_! = (\Sigma \times \mathbb R \times G)/\mathbb Z \to P = \Sigma
  \times \mathbb T,
\end{equation*}
where the $\mathbb Z$-action is generated by the diffeomorphism
prescribed, on $T$-points, by $(s, t, h) \mapsto (s, t+1, gh)$.  The
map $f_!\colon Q_! \to \hat X$ is induced by the $\mathbb Z$-invariant
map $(s,t,h)\mapsto f(s,e)\cdot h$.  The automorphism of $Q_!$
determined by the $G$-component of $f_!$ can be expressed as
\begin{math}
  (s,t,h) \mapsto (s,t,gh).
\end{math}

\section{The Chern character for global quotients}
\label{sec:chern-char-glob}

In this section, we show how to recover, in terms of dimensional
reduction of field theories, the delocalized Chern character of
\textcite{MR928402} (and, before them, \textcite{MR0461489}),
concerning the case of a finite group $G$ acting on a manifold $X$.

We start by briefly recalling the classical construction of $\ch_G$ in
section~\ref{sec:baum-connes-chern}.  On the field theory side, we can
associate to each vector bundle with connection $V$ on an orbifold
$\mathfrak X$ a field theory $E_V \in 1\vert1\EFT(\mathfrak X)$.  For
the sake of brevity, we will only describe, in section
\ref{sec:parall-transp-field}, the partition function of this theory,
denoted $Z_V \in C^\infty(\mathfrak K(\mathfrak X))$.  Finally, in
section \ref{sec:proof-global-quot-ch} we prove
theorem~\ref{thm:global-quot-ch}.

\subsection{The Baum--Connes Chern character}
\label{sec:baum-connes-chern}

As before, we write $\hat X = \{(x,g) \in X \times G \mid xg = x\} =
\coprod_{g\in G} X^g$.  The equivariant Chern character is a ring
homomorphism
\begin{equation}
  \label{eq:6}
  \ch_G\colon K^i_G(X) \to H_G^i(\hat X; \mathbb C)
  \cong \left[\bigoplus\nolimits_{g\in G}H^i(X^g; \mathbb C)\right]^G.
\end{equation}
Here, $i \in \mathbb Z/2$ and ordinary cohomology is $\mathbb
Z/2$-graded.  We recall that the equivariant ordinary cohomology of
$\hat X$ with coefficients in $\mathbb C$ can be identified with the
invariants in its nonequivariant cohomology; this can be deduced from
the Serre spectral sequence for the fibration $EG \times_G X \to BG$
using the fact that the integral reduced cohomology of a finite group
is torsion.

For each $g \in G$, we define the homomorphism $\ch_g\colon K^i_G(X)
\to H^i(X^g;\mathbb C)$ as the composition
\begin{equation*}
  K^i_G(X) \to K^i_{\langle g\rangle}(X^g) \cong K^i(X^g) \otimes R(\langle
  g\rangle) \xrightarrow{\ch \otimes \tr_g}
  H^i(X^g;\mathbb Q) \otimes_{\mathbb Q} \mathbb C.
\end{equation*}
(The middle isomorphism is due to the fact that the action of the
cyclic group $\langle g\rangle$ generated by $g$ on $X^g$ is trivial;
$\ch$ denotes the usual, nonequivariant Chern character, and $\tr_g$
assigns to any representation of $\langle g\rangle$ the trace of the
operator $g$.)  Finally, we let $\ch_G\colon K^i_G(X) \to H_G^i(\hat
X; \mathbb C)$ be the direct sum of all $\ch_g$ via the identification
\eqref{eq:6}.

Concretely, the effect of $\ch_g$ on the $K$-theory class represented
by a $G$-equivariant vector bundle $V \to X$ is the following.  For
each $x\in X^g$, the fiber $V_x$ is a representation of the cyclic
group generated by $g$.  Let $\lambda_1, \dots \lambda_r$ be distinct
eigenvalues, and $V_x^1, \dots, V_x^r$ the corresponding eigenspaces.
Each $\lambda_i$ is a $\abs{g}$-root of unity, so $V\vert_{X^g}$ can
be written as direct sum of vector bundles
\begin{equation*}
  V\vert_{X^g} = V^1 \oplus \cdots \oplus V^r.
\end{equation*}
Then
\begin{equation*}
  \ch_g(V) = \sum \lambda_i \ch(V^i) \in H^{\mathrm{ev}}(X^g;\mathbb C)
\end{equation*}
and
\begin{equation}
  \label{eq:9}
  \ch_G(V) = \oplus_{g \in G} \ch_g(V) \in
  \left[\bigoplus\nolimits_{g\in G}H^{\mathrm{ev}}(X^g; \mathbb C)\right]^G.
\end{equation}

This is the correct equivariant extension of the Chern character in
the sense that, for compact $X$, $\ch_G$ induces an isomorphism after
tensoring with $\mathbb C$.  Note that, in light of the
Atiyah--Segal completion theorem \cite{MR0259946}, the so-called
delocalized cohomology ring $H^*_G(\hat X;\mathbb C)$ is a stronger
invariant than ordinary equivariant cohomology of $X$.  For instance,
taking $X = \mathrm{pt}$, $\ch_G$ provides an identification between
the complexified representation ring $K^0_G(\mathrm{pt}) \otimes
\mathbb C$ and the ring of characters of $G$.  On the other hand,
$\tilde H^*_G(\mathrm{pt}; \mathbb C) = 0$.

\subsection{Parallel transport and field theories}
\label{sec:parall-transp-field}

Let $\mathfrak X$ be a stack on $\SM$ and $V\colon \mathfrak X \to
\mathrm{Vect}^{\nabla}$ a vector bundle with connection.  (If
$\mathfrak X$ is representable by a manifold, then, by the Yoneda
lemma, a fibered functor $\mathfrak X \to \Vect^\nabla$ is just a
vector bundle with connection on $\mathfrak X$.  In general, $V$
provides a natural assignment, to each $S$-point $x\colon
S\to\mathfrak X$, of a vector bundle with connection $V_x$ on $S$.)
Then we would like to construct a field theory $E_V \in
1\vert1\EFT(\mathfrak X)$ using parallel transport along superpaths in
$\mathfrak X$.  Roughly speaking, this EFT assigns to a superpoint
$x\colon\spt \to \mathfrak X$ the fiber $V_x$, and to a bordism
between those the super parallel transport map constructed by
\textcite{MR2407109}.  It is then part of the conjecture of Stolz and
Teichner on the relation between $1\vert1$-EFTs and $K$-theory that,
for reasonable $\mathfrak X$, the field theory above corresponds to
the $K$-theory class represented by $V$.

A construction of the field theory $E_V$ necessitates a longer
discussion on the bordism category $1\vert1\EBord(\mathfrak X)$, and,
for $\mathfrak X$ an orbifold, is given in a subsequent paper
\cite{arXiv:1801.03016}.  In any case, we are presently only
interested in its partition function, that is, the function $Z_V \in
C^\infty(\mathfrak K(\mathfrak X))$ obtained by restricting $E_V$ to
closed, connected bordisms.  (Note that the reduced field theory
$\mathrm{red}(E_V)$ relevant for theorem~\ref{thm:global-quot-ch} only
depends on $Z_V$.)  The partition function admits a straightforward
description independent of the details of the construction of the
full EFT; the goal of this subsection is to provide a detailed
construction of the functor
\begin{equation*}
  V \in \Vect^\nabla(\mathfrak X)
  \longmapsto Z_V \in C^\infty(\mathfrak K(\mathfrak X)).
\end{equation*}

We start recalling Dumitrescu's super version of parallel transport,
modified to better fit our conventions and perspective.  Fix, as
above, a vector bundle with connection $V\colon \mathfrak X \to
\Vect^{\nabla}$ and a map $\psi\colon S \times \mathbb R^{1\vert1} \to
\mathfrak X$, which we think of as an $S$-family of superpaths.  Fix
also sections $a, b\colon S \to S \times \mathbb R^{1\vert1}$, which
we think of as specifying endpoints of the superinterval $[a, b]
\subset S \times \mathbb R^{1\vert1}$.  The composition
\begin{equation*}
  S \times \mathbb R^{1\vert1}
  \overset \psi\longrightarrow \mathfrak X
  \overset V\longrightarrow \Vect^\nabla
\end{equation*}
determines a vector bundle $V_\psi$ over $S \times \mathbb
R^{1\vert1}$ with connection
\begin{equation*}
  \nabla\colon C^\infty(S \times \mathbb R^{1\vert1}; V_\psi)
  \to \Omega^1(S \times \mathbb R^{1\vert1}; V_\psi).
\end{equation*}
Further restricting via $a$ and $b$ gives us vector bundles
$V_{\psi(a)}$ and $V_{\psi(b)}$ over $S$.  Now, we define a section
$s$ of $V_\psi$ to be \emph{parallel} if $\nabla_D s = 0$.  (Here, as
in appendix~\ref{sec:eucl-superm}, $D = \partial_\theta - \theta
\partial_t$ is the standard left-invariant vector field on $\mathbb
R^{1\vert1}$.  Since $D^2 = -\partial_t$, this can be though of as a
``half-order'' differential equation.)  It can be shown
\cite[proposition~3.1]{MR2407109} that any section $s_a$ of
$V_{\psi(a)}$ determines a unique parallel section $s$ of $V_\psi$.
Parallel transport is then the linear map
\begin{equation*}
  \mathrm{SP}(\psi, a, b)\colon V_{\psi(a)} \to V_{\psi(b)}
\end{equation*}
obtained by restricting to $V_{\psi(b)}$ the parallel section with
given value on $V_{\psi(a)}$.

The main properties of super parallel transport are established in
\cite[theorem 3.5]{MR2407109}.  We recall them here for the
convenience of the reader.
\begin{enumerate}
\item[(SP1)] If a map $F\colon S' \times \mathbb R^{1\vert1} \to S
  \times \mathbb R^{1\vert1}$ covering $f\colon S' \to S$ is
  conformal, that is, preserves the distribution generated by $D$,
  and, moreover, $F(a') = a \circ f$ and $F(b') = b \circ f$, then
  \begin{equation*}
    \mathrm{SP}(F^*\psi, a', b') = f^*\mathrm{SP}(\psi, a, b).
  \end{equation*}
\item[(SP2)] Given $a, b, c\colon S \to \mathbb R^{1\vert1}$, we have
  \begin{equation*}
    \mathrm{SP}(\psi, a, c) = \mathrm{SP}(\psi, b, c)
    \circ \mathrm{SP}(\psi, a, b).
  \end{equation*}
\end{enumerate}
Property (SP1) encapsulates both the fact that parallel transport
depends smoothly on the input data $\psi$, $a$, $b$, and that it is
invariant under conformal reparametrization of superpaths (and in
particular under the Euclidean reparametrizations we are concerned
with in this paper).  Note that it is \emph{not} invariant under
arbitrary reparametrizations.  Property (SP2) is the expected
compatibility with gluing of superintervals.

\begin{remark}
  Dumitrescu describes parallel transport with respect to both $D$ and
  its right-invariant counterpart $\partial_\theta +
  \theta\partial_t$, emphasizing the latter
  \cite[remark~3.3]{MR2407109}.  Because of the way we set up super
  Euclidean structures, we must work with $D$-parallel transport.
  This leads to different sign conventions, for instance in the proof
  of proposition~\ref{prop:10}.  The second half of Dumitrescu's paper
  is concerned with the more subtle notion of a parallel transport
  operation for Quillen superconnections; we will not deal with this
  case here.
\end{remark}

We now return to our task of constructing, out of $V \in
\mathrm{Vect}^\nabla(\mathfrak X)$, a function $Z_V$ on $\mathfrak
K(\mathfrak X)$.  The idea is to associate, to each $S$-point $(K,
\psi\colon K \to \mathfrak X)$ of $\mathfrak K(\mathfrak X)$, the
supertrace of the holonomy along $K$.  To make this precise, note that,
by proposition~\ref{prop:3}, it suffices to restrict to $K$ of the
form $K_l = (S \times \mathbb R^{1\vert1})/\mathbb Zl$ for some length
parameter $l\colon S \to \mathbb R^{1\vert1}_{>0}$.  With a slight
abuse of notation, we still write $\psi$ for the induced periodic map
$S \times \mathbb R^{1\vert1} \to \mathfrak X$, so that $V_{\psi(0)} =
V_{\psi(l)}$.  Finally, we set
\begin{equation}
  \label{eq:8}
  Z_V(K_l, \psi) = \str(\mathrm{SP}(\psi, 0, l)) \in C^\infty(S).
\end{equation}

\begin{proposition}
  \label{prop:7}
  For any morphism $F\colon (K_{l'}, \psi') \to (K_l, \psi)$ covering
  $f\colon S' \to S$ in $\mathfrak K(\mathfrak X)$, we have
  $Z_V(K_{l'}, \psi') = f^*Z_V(K_l, \psi)$.  Therefore, \eqref{eq:8}
  indeed defines a function $Z_V \in C^\infty(\mathfrak K(\mathfrak
  X))$.
\end{proposition}

\begin{proof}
  Again by proposition~\ref{prop:3}, any isometry $F\colon K_{l'} \to
  K_l$ lifts to a fiberwise isometry $\tilde F\colon S' \times \mathbb
  R^{1\vert1} \to S \times \mathbb R^{1\vert1}$ over the same $f\colon
  S' \to S$.  If $\tilde F(0) = 0$ and $\tilde F(l') = l$, then the
  proposition follows immediately from (SP1).  Thus, replacing $K_l$
  by its pullback to $S'$ if needed, we may assume that $S' = S$ and
  $f$ is the identity.

  Now, $\tilde F$ is determined by a map $S' \to \mathbb Z/2 \ltimes
  \mathbb R^{1\vert1}$; simple calculations, done separately for the
  case of flips and translations, show that $\mathrm{SP}(\psi, 0,
  \tilde F(0)) = \mathrm{SP}(\psi, l, \tilde F(l'))$.  Thus, using
  (SP2) and the vanishing of supertrace on commutators,
  \begin{align*}
    Z_V(K_l,\psi)
    & = \str(\mathrm{SP}(\psi, \tilde F(0),l) \circ
      \mathrm{SP}(\psi, 0, \tilde F(0))) \\
    & = \str(\mathrm{SP}(\psi, 0, \tilde F(0)) \circ
      \mathrm{SP}(\psi, \tilde F(0), l)) \\
    & = \str(\mathrm{SP}(\psi, l, \tilde F(l')) \circ
      \mathrm{SP}(\psi, \tilde F(0), l)) \\
    & = \str(\mathrm{SP}(\psi, \tilde F(0), \tilde F(l'))).
  \end{align*}
  Finally, using (SP1), we get
  \begin{equation*}
    Z_V(K_{l'}, \psi')
    = Z_V(K_{l'}, F^*\psi) 
    = \str(\mathrm{SP}(\psi, \tilde F(0), \tilde F(l')))
    = Z_V(K_l, \psi),
  \end{equation*}
  which finishes the proof.
\end{proof}

The construction of $Z_V$ is clearly natural in $V$, and thus defines
a functor $Z\colon \mathrm{Vect}^\nabla(\mathfrak X) \to
C^\infty(\mathfrak K(\mathfrak X))$.

\subsection{Proof of theorem \ref{thm:global-quot-ch}}
\label{sec:proof-global-quot-ch}

As before, fix $V\colon X \sslash G \to \mathrm{Vect}^{\nabla}$.  This
map classifies a $G$-equivariant vector bundle over $X$, which we
still call $V$, with a $G$-invariant connection $\nabla$.  To get
started, we need to describe the pullback of $V$ to a supercircle over
$X\sslash G$.

\begin{proposition}
  \label{prop:1}
  Fix a supercircle $\psi\colon K \to X\sslash G$ and denote by
  $\pi\colon Q \to K$ and $f\colon Q \to X$ the principal $G$-bundle
  and $G$-equivariant map classified by $\psi$.  Then there is a
  natural connection-preserving isomorphism of vector bundles
  \begin{equation*}
    \xymatrix{(f^*V)/G \ar[r] \ar[d] & \psi^*V \ar[d] \\ Q/G \ar@{=}[r] &
      K.}
  \end{equation*}
\end{proposition}

\begin{proof}
  Consider the diagram
  \begin{equation*}
    \xymatrix{Q \times_{K} Q \ar@<.6ex>[r] \ar@<-.6ex>[r]
      & Q \ar[r]^f \ar[d]^\pi & X \ar[d]^x\\
      & K \ar[r]^\psi & X\sslash G \ar[r]^V & \mathrm{Vect}^\nabla.}
  \end{equation*}
  Here, $x\colon X \to X\sslash G$ is the standard atlas and hence
  $V\circ x$ classifies the vector bundle with connection $V\to X$.
  Notice that the square $2$-commutes.  In fact, the top-right
  composition $Q \to X\sslash G$ classifies the trivial $G$-bundle $Q
  \times G \to Q$, while the left-bottom composition classifies the
  $G$-bundle $\pi^* Q \to Q$ (together with the corresponding
  equivariant maps into $X$ induced by $f$), and these two $Q$-points
  of $X\sslash G$ are isomorphic.

  Now, the composition $V\circ x \circ f$ classifies the vector bundle
  $f^*V \to Q$, and the $G$-equivariance information provides descent
  data for the covering $Q \times_K Q \cong Q \times G
  \rightrightarrows Q \to K$.  The descended vector bundle with
  connection can be described explicitly as $(f^*V)/G \to K$.  Thus,
  $2$-commutativity of the square above and the uniqueness property of
  descent provide a canonical isomorphism $\psi^*V \cong (f^*V)/G$.
\end{proof}

Our goal now is to identify $\mathrm{red}(Z_V) \in C^\infty(\mathfrak
B(\hat X\sslash G))$, the dimensional reduction of $Z_V \in
C^\infty(\mathfrak K(\mathfrak X))$, with a ($G$-invariant, even,
closed) differential form on $\hat X$, following the identifications
of theorem~\ref{thm:0|1-EFT-orbifold}.  To do so, we need to consider
the versal $\Pi T\hat X$-family $\Sigma_{\mathrm{versal}} \in
\mathfrak B(\hat X\sslash G)$
\begin{equation}
  \label{eq:7}
  \begin{gathered}
    \xymatrix{\Pi T\hat X \times \mathbb
      R^{0\vert1}\ar[r]^-{\operatorname{ev}} \ar[d]
      & \hat X \xrightarrow{\hat x} \hat X\sslash G\\
      \Pi T\hat X}
  \end{gathered}
\end{equation}
and calculate the smooth function on the parameter manifold $\Pi T\hat
X$ assigned to it via $\mathrm{red}(Z_V)$.  Here, $\hat x$ denotes the
usual atlas.

\begin{proposition}
  We have the following identity in $C^\infty(\Pi T\hat X)$:
  \begin{equation*}
    \mathrm{red}(Z_V)(\Sigma_{\mathrm{versal}}) = Z_V(K,Q,f),
  \end{equation*}
  where $(K, Q, f) \in \mathfrak K(\mathfrak X)$ is as defined below.
\end{proposition}

\begin{proof}
  This is an exercise in chasing through the definition of dimensional
  reduction.  The first step is to pick a preimage of
  $\Sigma_{\mathrm{versal}}$ via $\mathcal P$.  Such a preimage is
  obtained by adding to \eqref{eq:7} the trivial principal $\mathbb
  T$-bundle with standard connection over $\Pi T\hat X \times \mathbb
  R^{0\vert1}$.  The second step is to map that gadget to $\mathfrak
  B^{\mathbb R\sslash\mathbb Z}(\hat X \sslash G)$ via $\mathcal Q$.
  From the considerations at the end of
  section~\ref{sec:global-quotients}, it follows that the resulting
  $\Pi T\hat X$-family, once restricted to $\Pi TX^g \subset \Pi T
  \hat X$, comprises the following data:
  \begin{enumerate}
  \item the family of Euclidean $0\vert1$-manifolds $\Sigma = \Pi TX^g
    \times \mathbb R^{0\vert1} \to \Pi TX^g$,
  \item the trivial $\mathbb T$-bundle $P_g = \Pi TX^g\times \mathbb
    R^{0\vert1} \times \mathbb T \to \Sigma$, with the standard
    connection form $\omega = dt - \theta d\theta$,
  \item the principal $G$-bundle $Q_g = (\Pi TX^g \times \mathbb
    R^{1\vert1} \times G)/\mathbb Z \to P_g$, where the $\mathbb
    Z$-action is generated by the map described on $S$-points by
    \begin{equation*}
      (x,t,h) \in (\Pi TX^g \times \mathbb R^{1\vert1}
      \times G)_S \mapsto (x, 1\cdot t, gh),
    \end{equation*}
  \item the map $f_g\colon Q_g \to \hat X \subset X \times G$ given by
    $(x, t, h) \mapsto (\operatorname{ev}(t_1,x)\cdot h, h^{-1}gh)$,
    which is well defined since $\operatorname{ev}(t_1,x)$ lies in
    $X^g$.
  \end{enumerate}
  Finally, mapping to $\mathfrak K(\mathfrak X)$ via $\mathcal R$
  results in the $\Pi T \hat X$-family $(K, Q, f) \in \mathfrak
  K(\mathfrak X)$ determined, over $\Pi T X^g$, by the data of $P_g$,
  seen as a supercircle, and the map $P_g \to \mathfrak X$ determined
  by $Q_g$ and $\pr_1 \circ f_g$.  By construction, the equation in
  the statement of the proposition holds true.
\end{proof}

Our next task is to compute $Z_V(K, Q, f)$; this is, by definition,
the supertrace of the holonomy (around $K$) of the pullback of $V$ by
the map $\psi\colon K \to X\sslash G$ determined by $(Q, f)$.
Proposition~\ref{prop:1} identifies that pullback of $V$ with the
vector bundle with connection $W = (f^*V)/G \to K$.

\begin{proposition}
  \label{prop:10}
  On $\Pi TX^g$, the supertrace of the holonomy of $W = (f^*V)/G$
  around $K$ is a differential form representative of
  $\mathrm{ch}_g(V)$.
\end{proposition}

Here, of course, we employ the usual identification $C^\infty(\Pi
TX^g) \cong \Omega^*(X^g)$.

\begin{proof}
  Consider the standard superpath $c\colon \Pi TX^g \times \mathbb
  R^{1\vert1} \to \Pi TX^g \times \mathbb T^{1\vert1} \subset K$ with
  endpoints $i_t\colon \Pi TX^g \to \Pi TX^g \times \mathbb
  R^{1\vert1}$, $x \mapsto (x,t)$, for $t = 0,1$, and denote by
  $\operatorname{SP}\colon c_0^*W \to c_1^*W$ the parallel transport
  operator along that superinterval.  There is a slight subtlety to
  notice here.  Since the maps $c_0 = c \circ i_0$ and $c_1 = c \circ
  i_1$ are equal, $c_0^*W$ and $c_1^*W$ are the same vector bundle,
  but the correct way to identify them (for the purposes of computing
  the holonomy) is via the action of $g$.  Indeed, let us form the
  pullback of principal bundles
\begin{equation*}
  \xymatrix{
    \tilde Q \ar[d] \ar[r] \ar@/^{3.5ex}/[rr]^(0.8){\tilde f} & Q \ar[d] \ar[r]_-f
    &X\\
    \Pi TX^g \times \mathbb R^{1\vert1} \ar[r]^-c & K &
  }    
\end{equation*}
where $\tilde Q = \Pi TX^g \times \mathbb R^{1\vert1} \times
G$.  Then the pullback $c^*W$ can be identified with the restriction
of the pullback of $V$ to the identity section of $\tilde Q$,
\begin{equation*}
  c^*W \cong (\tilde f^*V)/G \cong (\tilde f^*V)\vert_{\Pi TX^g \times
    \mathbb R^{1\vert1} \times \{ e \}},
\end{equation*}
so we identify
\begin{equation*}
  c^*_0W_x = \tilde f^*V_{(x, 0, e)} = \tilde f^*V_{(x,1,g)}
   \xrightarrow{g^{-1}} \tilde f^*V_{(x,1,e)} = c_1^*W_x.
\end{equation*}

We write, as before, $V\vert_{X^g} = V^1 \oplus \cdots \oplus V^r$ as
a direct sum of eigenspaces for eigenvalues $\lambda_1, \dots,
\lambda_r$, with connection $\nabla_i$ on each component.  Since
$\tilde f\vert_{\Pi TX^g \times \mathbb R^{1\vert1} \times \{ e \}}$
takes values in $X^g$, this induces a similar decomposition of $c^*W$
into a sum of vector bundles $\tilde W^i$ with connection.  We are
finally ready to invoke the calculations of Dumitrescu recovering the
usual (nonequivariant) Chern character in terms of parallel transport.
Denoting by $\operatorname{SP}_i\colon \tilde W^i\vert_{\Pi TX^g
  \times 0} \to \tilde W^i\vert_{\Pi TX^g \times 1}$ the super
parallel transport for one unit of time on each $\tilde W^i$, the main
theorem of \cite{arXiv:1202.2719} states that $\operatorname{SP}_i
=\exp(\nabla_i^2)$, so that
\begin{equation*}
  \ch(\nabla_i) =
  \str(\exp(\nabla_i^2)) = \str(\operatorname{SP}_i).
\end{equation*}

The holonomy endomorphism $H\colon c_0^*W \to c_0^*W$ can be expressed
as the composition
\begin{equation*}
  c_0^*W
  = \bigoplus\nolimits_i \tilde W^i\vert_{\Pi TX^g \times 0}
  \xrightarrow{\oplus_i \operatorname{SP}_i}
   \bigoplus\nolimits_i \tilde W^i\vert_{\Pi TX^g \times 1} =
  c_1^*W \overset g \longrightarrow c_0^*W
\end{equation*}
and we conclude that
\begin{equation*}
  \str(H)
  = \sum_{1 \leq i \leq r} \str (g\operatorname{SP}_i)
  = \sum_{1 \leq i \leq r} \lambda_i \str (\operatorname{SP}_i)
  = \sum_{1 \leq i \leq r} \lambda_i \ch(\nabla_i)
\end{equation*}
is a differential form representative of $\ch_g(V)$.
\end{proof}

Finally, recall that the differential form on $\hat X = \amalg_{g \in
  G} X^g$ associated to the field theory $\operatorname{red}(Z_V)$ is
the form corresponding to the function
$\operatorname{red}(Z_V)(\Sigma_{\mathrm{versal}})$ on $\Pi T\hat X$.
In particular, by theorem~\ref{thm:0|1-EFT-orbifold}, this form is
$G$-invariant and represents an element of $H^{\mathrm{ev}}_G(\hat X;
\mathbb C)$.  By the above proposition and \eqref{eq:9}, this element
is $\ch_G(V)$.  This finishes the proof that the diagram of
theorem~\ref{thm:global-quot-ch} commutes.  The claim that the
vertical arrow in that diagram induces a bijection after passing to
concordance classes is the content of
theorem~\ref{thm:0|1-EFT-orbifold}.

\appendix

\section{Group actions on stacks}
\label{sec:group-actions-stacks}

We briefly review the definitions of group action on a stack and
quotient of a stack, following \textcite{MR2125542} and
\textcite{arXiv:1206.5603}, and then prove a useful lemma
(proposition~\ref{prop:6}).  Note that limits and colimits here are
always taken in the sense of bicategories.  These are often called
$2$-(co)limits, bi(co)limits or homotopy (co)limits.

\subsection{Basic definitions}
\label{sec:basic-definitions-action}

Let $\mathfrak X$ be a groupoid fibration over a site $\mathfrak S$
and $G$ a strict monoid object in the $2$-category of fibrations over
$\mathfrak S$.  We denote by $m\colon G \times G \to G$ and $1\colon
\mathrm{pt} \to G$ the multiplication law and unit map of $G$.  Then
we define a \emph{(left) action of $G$ on $\mathfrak X$} to be a map
of groupoid fibrations $\mu\colon G \times \mathfrak X \to \mathfrak
X$ together with (necessarily invertible) $2$-morphisms $\alpha,
\mathfrak a$ as in the diagram below.
\begin{equation*}
  \xymatrix{G \times G \times \mathfrak X \ar[r]^-{m \times
      \mathrm{id}} \ar[d]_{\mathrm{id} \times \mu}
    & G \times \mathfrak X \ar[d]^\mu \\
    G \times \mathfrak X \ar[r]^\mu \ar@{=>}[ur]^\alpha& \mathfrak X}
  \qquad
  \xymatrix{G \times \mathfrak X \ar[r]^\mu & \mathfrak X \\ \mathfrak X
    \ar[u]^{1 \times \mathrm{id}} \ar[ur]_{\mathrm{id}}^{}="a" \ar@{=>}_{\mathfrak a}"1,1";"a"}
\end{equation*}
In formulas, given an object $x \in \mathfrak X_S$ and $g,h \in G_S$,
and using a dot to denote the group action, we are given natural
isomorphisms
\begin{equation*}
  \alpha^x_{g, h}\colon g \cdot (h\cdot x) \to (gh) \cdot x, \qquad
  \mathfrak a^x\colon 1\cdot x \to x.
\end{equation*}
This data is required to satisfy compatibility conditions that bear
some resemblance to the axioms of a monoidal category.  Firstly, a
kind of pentagon identity relating the different ways in which the
action of three group elements $g, h, k \in G_S$ can be associated:
\begin{equation*}
  \alpha^x_{g, hk} \circ g \cdot \alpha^x_{h,k} = \alpha^x_{gh, k}
  \circ \alpha^{g\cdot x}_{g,h}.
\end{equation*}
Second, a condition on the two ways of associating the action of the
unit and another group element:
\begin{equation*}
  g \cdot \mathfrak a^x = \alpha^x_{g,1} \text{ and } \mathfrak a^{g\cdot x} =
  \alpha^x_{1, g}.
\end{equation*}
It seems appropriate to call $\alpha$ and $\mathfrak a$ the associator
and unitor for the action, in analogy to the terminology used in the
theory of monoidal categories.  We say the action is strict if
$\alpha$, $\mathfrak a$ are both the identity.

Now, suppose we are given fibrations with $G$-action $(\mathfrak X,
\mu, \alpha, \mathfrak a)$ and $(\mathfrak Y, \nu, \beta, \mathfrak
b)$.  Then a \emph{$G$-equivariant map} between them is a morphism of
fibrations $f\colon \mathfrak X \to \mathfrak Y$ together with a
$2$-morphism
\begin{equation*}
  \xymatrix{G \times \mathfrak X \ar[r]^\mu \ar[d]_{\mathrm{id} \times f} & \mathfrak X
  \ar[d]^f\\ G \times \mathfrak Y \ar[r]^\nu \ar@{=>}[ur]^{\rho}&\mathfrak Y}
\end{equation*}
satisfying the following compatibility condition: for each $x \in
\mathfrak X_S$ and $g, h \in G_S$, we have
\begin{equation*}
  f(\alpha^x_{g,h}) \circ \rho_g^{h\cdot x} \circ g\cdot \rho^x_h
  = \rho^x_{gh} \circ \beta^{f(x)}_{g,h} \text{ and } f(\mathfrak
  a^x) \circ \rho^x_1 = \mathfrak b^{f(x)}.
\end{equation*}
We will call $\rho$ the equivariance datum.  Finally, a
\emph{$G$-equivariant $2$-morphism} between morphisms $(f, \rho)$,
$(f', \rho')$ as above is defined to be a $2$-morphism $\xi\colon f
\to f'$ between the underlying fibered functors which is compatible
with $\rho$, $\rho'$ in the sense that
\begin{equation*}
  \rho'^x_g \circ g \cdot \xi^x = \xi^{g\cdot x} \circ \rho^x_g
\end{equation*}
for any $x \in \mathfrak X_S$, $g \in G_S$.

In terms of pasting diagrams, the conditions on $\rho$ are expressed
by the commutativity of the cube whose two halves are depicted below,
\begin{equation*}
  \xymatrix@C=1.2em{
    G\times G\times \mathfrak X \ar[rd]^{m\times\id}\ar[dd]
    &&\\
    & G\times \mathfrak X \ar[rr]^\mu\ar[dd] && \mathfrak X\ar[dd]_{}="a"\\
    G\times G\times \mathfrak Y\ar[rd]^{m\times\id}\ar[rr]^(0.7){\id\times\nu}|(0.535)\hole
    &&G\times\mathfrak Y\ar[rd]_(0.4)\nu|(0.58)\hole\ar@{=>}[ld]_\beta\\
    &G\times\mathfrak Y\ar[rr]^\nu^(0.7){}="b"\ar@{=>}"b";"a"^(0.6)\rho&&\mathfrak Y
  }
  \,
  \xymatrix@C=1.2em{
    G\times G\times \mathfrak X \ar[rr]^{\id\times\mu}_{}="a" \ar[rd]_(0.4){m\times\id}\ar[dd]
    &&G \times\mathfrak X \ar[rd]^\mu\ar[dd]|{\hole}
    \ar@{=>}[dl]_\alpha\\
    & G\times \mathfrak X \ar[rr]^(0.4)\mu && \mathfrak X\ar[dd]\\
    G\times G\times \mathfrak Y\ar[rr]^{\id\times\nu}\ar@{=>}[];"a"_(0.4){\id\times\rho}|(0.68)\hole
    &&G\times\mathfrak Y\ar[rd]^\nu\ar@{=>}[ur]^\rho\\
    &&&\mathfrak Y
  }
\end{equation*}
and commutativity of the prism
\begin{equation*}
  \xymatrix@R=1em{
    & G\times\mathfrak X\ar[dd]|\hole\ar[rd]^\mu_{}="d"\\
    \mathfrak X \ar[ur]^{1\times\id}\ar[rr]_(0.3)\id^(0.6){}="c"\ar[dd]
    \ar@{=>}"d";"c"_{\mathfrak a}&& \mathfrak X\ar[dd]\\
    & G\times\mathfrak Y \ar[rd]^\nu_{}="b"\ar@{=>}[ur]_{\rho}\\
    \mathfrak Y \ar[rr]_\id^{}="a"\ar[ur]^{1\times\id}&&\mathfrak Y.
    \ar@{=>}"b";"a"_{\mathfrak b}
  }
\end{equation*}
Here, all vertical maps are products of $f$ and the identity of $G$.
The condition on $\xi$ is the commutativity of the following diagram.
\begin{equation*}
  \xymatrix@R=4em@C=4em{
    G \times \mathfrak X \ar[r]^-\mu_-{}="b" \ar@/_{3ex}/[d]_{\mathrm{id}
      \times f}^{}="x" \ar@/^{3ex}/[d]^{\mathrm{id} \times f'}^(0.4){}="a"_{}="y"\ar@{=>}"x";"y"^{\id\times\xi}\ar@/^{.5ex}/@{=>}"a";"b"_(0.7){\rho'} & \mathfrak X
  \ar@/_{3ex}/[d]_f^{}="w"_{}="d" \ar@/_{-3ex}/[d]^{f'}_{}="z"\ar@{=>}"w";"z"^{\xi}\\ G \times \mathfrak Y
  \ar[r]_-\nu^(0.4){}="c"\ar@/^{-.5ex}/@{=>}"c";"d"_{\rho}  &\mathfrak Y}
\end{equation*}

We are mostly interested in the case where $G$ is a (representable)
sheaf of groups, but we will also consider the group stack
$\mathrm{pt}\sslash\mathbb Z$.  Note that a strict action of
$\mathrm{pt}\sslash\mathbb Z$ on a stack $\mathfrak X$ is precisely
the data of an automorphism of $\mathrm{id}_{\mathfrak X}$, i.e., a
natural choice of automorphism for each object of $\mathfrak X$.  For
instance, the inertia stack $\Lambda\mathfrak X$ comes with a
canonical $\mathrm{pt}\sslash\mathbb Z$-action.  We will also make use
of a $2$-categorical model for the circle group to be denoted $\mathbb
R\sslash\mathbb Z$.  It is presented by the Lie $2$-group $\mathbb
Z\times\mathbb R\rightrightarrows\mathbb R$ (the transport groupoid of
the $\mathbb Z$-action on $\mathbb R$) endowed with the multiplication
map determined by the group structures on the spaces of objects and
morphisms, and unit $0 \in \mathbb R$.  At the Lie $2$-group level,
there are evident strict homomorphisms
\begin{equation*}
  \mathbb T \leftarrow \mathbb R\sslash\mathbb Z \to \mathrm{pt}\sslash\mathbb Z.
\end{equation*}
The left map gives us an equivalence of group stacks, but in concrete
situations it may be more convenient to consider one model or the
other.

\subsection{Quotient stacks of $G$-stacks}
\label{sec:quot-stacks-stacks}

Let $\mathfrak X$ be a stack endowed with a left action of a sheaf of
groups $G$.  Then we define a new stack $G\backsslash\mathfrak X$
whose $S$-points are given by a left $G$-torsor $P \to S$ together
with a $G$-equivariant map $\psi\colon P\to \mathfrak X$; a morphism
$(P', \psi') \to (P, \psi)$ covering $f\colon S' \to S$ is given by a
diagram
\begin{equation*}
  \begin{gathered}
    \xymatrix@R-1.5em{ P'
      \ar[dd]\ar[rd]^\Phi
      \ar@/^3ex/[rrd]^{\psi'}_{}="a"&&\\
      & P \ar[dd]\ar[r]^{\psi}  \ar_-\xi@{<=}"a"& \mathfrak X\\
      S' \ar[rd]^f &&\\
      & S &}
  \end{gathered}
\end{equation*}
where $\Phi$ is a map of $G$-torsors and $\xi$ an equivariant
$2$-morphism.

There is a faithful functor $i\colon \mathfrak X\to
G\backsslash\mathfrak X$ sending $x\colon S\to \mathfrak X$ to the
$S$-point of $G\backsslash\mathfrak X$ consisting of the trivial
$G$-torsor $G \times S \to S$ together with the $G$-equivariant map
\begin{equation*}
  \psi\colon G\times S \xrightarrow{\mathrm{id}\times x} G\times\mathfrak X
  \xrightarrow\mu \mathfrak X.
\end{equation*}
This makes the diagram below $2$-cartesian.
\begin{equation*}
  \begin{gathered}
    \xymatrix{G \times\mathfrak X \ar[r]^\mu\ar[d]_{\pr_2}&\mathfrak
      X\ar[d]^i\\
      \mathfrak X\ar[r]^i&G\backsslash\mathfrak X}
  \end{gathered}
\end{equation*}

Now, we can attempt to perform the construction of a transport
groupoid $G\times\mathfrak X \rightrightarrows \mathfrak X$ internally
in the $2$-category of stacks.  For this to work, we need to define
internal categories with the appropriate degree of weakness (e.g., if
the action is not strictly unital, the same must be allowed of our
internal categories).  In any case, it is clear that we get a
``nerve'', that is, an augmented (weak) simplicial object
\begin{equation}
  \label{eq:23}
  G\backsslash\mathfrak X \overset i\leftarrow \mathfrak X \leftleftarrows G\times\mathfrak X \threeleftarrows
   G \times G \times\mathfrak X\fourleftarrows\cdots
\end{equation}
Since the various compositions $G^n \times \mathfrak X \to
G\backsslash\mathfrak X$ are not equal, just isomorphic (with a
specified isomorphism), the augmentation depends, strictly speaking,
on a choice.  For definiteness, we take that to be the composition of
$i$ with the projection $\pr_{n+1}\colon G^n\times\mathfrak X
\to\mathfrak X$.

\begin{proposition}
  \label{prop:9}
  The above induces an equivalence of stacks
  \begin{equation*}
    G\backsslash \mathfrak X \xleftarrow j \colim\left( \mathfrak X
      \leftleftarrows G\times\mathfrak X \threeleftarrows G \times G \times\mathfrak X\fourleftarrows\cdots\right).
  \end{equation*}
\end{proposition}

The reader well versed on colimits of categories may be able to
interpret the discussion in sections 3.2 and 4.2 of Ginot and Noohi's
paper \cite{arXiv:1206.5603} as a proof, even though it does not use
the language of colimits.  In any case, we will provide our own
argument.  Before getting there, we give some background on (homotopy)
colimits in $\mathrm{Cat}$.  Given a diagram of small categories
$F\colon D \to \mathrm{Cat}$ indexed by a small $1$-category (with no
strictness requirements on $F$), we denote by $D \ltimes F$ the
Grothendieck construction.  It is the oplax colimit of $F$, meaning
that for each $C \in \mathrm{Cat}$, there is an equivalence between
the category of functors $D \ltimes F \to C$ and the category of lax
natural transformations $F \to \mathrm{const}_C$ and modifications
between them.  The colimit of $F$ is obtained by localizing $D \ltimes
F$ at the class of opcartesian morphisms.

Spelling out the above, the colimit can be described in terms of
generators and relations as follows.  We write $i$, $j$, etc., for
objects of $D$ and $A_i$, $A_j$ for their images via $F$; also, we use
the same notation both for a morphism $f\colon i\to j$ in $D$ and its
image $f\colon A_i \to A_j$.  To build $A = \colim_D A_i$, we start
with the disjoint union $\amalg_{i\in D} A_i$ and then freely adjoin
inverse morphisms
\begin{equation*}
  f_x\colon x \to f(x),\quad f_x^{-1}\colon f(x) \to x
\end{equation*}
for each $f\colon i\to j$ in $D$ and $x \in A_i$; finally, we impose a
number of natural relations, most notably
\begin{equation*}
  \left( x \xrightarrow\phi y \xrightarrow{f_y} f(y) \right)
  = 
  \left( x \xrightarrow{f_x} f(x) \xrightarrow{f(\phi)} f(y) \right),
\end{equation*}
where $\phi$ is a morphism in $A_i$, as well as its counterpart
involving $f_x^{-1}$, $f_y^{-1}$.  This process can be made precise
using the free category generated by a directed graph and congruences.
For more details, including the proof that this has the desired
universal property, see \textcite[chapter~4]{MR2229946}.

\begin{proof}[Proof of proposition~\ref{prop:9}]
  Colimits of stacks are obtained by taking colimits objectwise in
  $\mathfrak S$ and then stackifying.  Thus, it suffices to show that,
  for each $S \in \mathfrak S$,
  \begin{equation*}
    (G\backsslash \mathfrak X)_S \xleftarrow{j_S} \colim\left( \mathfrak X_S
      \leftleftarrows (G\times\mathfrak X)_S \threeleftarrows (G
      \times G \times\mathfrak X)_S \fourleftarrows \cdots\right)
  \end{equation*}
  gives an equivalence of the right-hand side with the full
  subgroupoid $(G\backsslash \mathfrak X)_S^{\mathrm{triv}}$ of the
  left-hand side involving only trivial $G$-torsors.  To simplify the
  argument, we assume, without loss of generality, that the
  $G_S$-action on $\mathfrak X_S$ is strict \cite[proposition
  1.5]{MR2125542}.

  Consider the functor $l\colon (G\backsslash\mathfrak
  X)_S^{\mathrm{triv}} \to \colim_n (G^n \times \mathfrak X)_S$
  prescribed by the following conditions.  First, on $\mathfrak X_S$,
  seen as a subgroupoid of both the domain (via $i\colon \mathfrak X_S
  \hookrightarrow (G\backsslash\mathfrak X)_S^{\mathrm{triv}}$) and
  codomain, $l$ is just the identity.  Second, to the morphism $x \to
  g\cdot x$ in $(G\backsslash\mathfrak X)_S^{\mathrm{triv}}$
  determined by $g \in G_S$, $l$ associates the morphism
  \begin{equation*}
    \mu_g^x\colon x \xrightarrow{\pr_2^{-1}} (g,x)
    \xrightarrow\mu g\cdot x
  \end{equation*}
  in the colimit groupoid.  To see that this is well defined and
  respects compositions, it suffices to check that the outer square of
  the following diagram in the colimit groupoid commutes, for any $g,h
  \in G_S$ and $\xi\colon g\cdot x \to y$ in $\mathfrak X_S$.
  \begin{equation*}
    \xymatrix{
      g\cdot x \ar[rrr]^\xi\ar[dd]^{\mu_h^{g\cdot x}} &&&
      y\ar[dd]^{\mu_h^y}\\
      & (h,g\cdot x) \ar[r]^{\id\times\xi}\ar[ul]_{\pr_2}\ar[dl]^{\mu}&(h,y)\ar[ur]^{\pr_2}\ar[dr]_{\mu}
      \\
      hg\cdot x \ar[rrr]^{h\cdot\xi} &&& h\cdot y
    }
  \end{equation*}
  This follows from the fact that each circuit traveling inside the
  square commutes.

  Now, the composition $j_S\circ l$ is equal to the identity, and we
  claim that the reverse composition is isomorphic to the identity.
  In fact, $l\circ j_S(g_1,\dots,g_n,x) = x$, and we define a natural
  transformation $u\colon \mathrm{id} \to l\circ j_S$ by
  \begin{equation*}
    u_{(g_1,\dots,g_n,x)} = \pr_{n+1}\colon (g_1,\dots,g_n,x) \to x.
  \end{equation*}
  Naturality with respect to those morphisms in the colimit groupoid
  which arise from morphisms in $(G^n\times\mathfrak X)_S$ is obvious.
  A general morphism arising from the indexing category
  $\Delta^{\mathrm{op}}$ is as in the left vertical arrow of the
  diagram below,
  \begin{equation*}
    \xymatrix{
      (g_1,\dots,g_n,x) \ar[rr]^-{\pr_{n+1}} \ar[dd]\ar[dr] &&
      x\ar[dd]^{\mu^x_{g_J}}\\
      &(g_J,x) \ar[ur]^{\pr_2}\ar[dr]_\mu&\\
      (g_{I_1}, \dots,g_{I_k}, g_J\cdot x) \ar[rr]^-{\pr_{k+1}}
      && g_J\cdot x
    }
  \end{equation*}
  where, $I_1,\dots,I_k,J \subset [n]$ are (possibly empty) disjoint
  and adjacent subsets whose union contains $n$, and
  $g_{\{ i_1,\dots,i_j \}} = g_{i_1}\dots g_{i_j}$.  Its image through
  $l\circ j_S$ is the right vertical arrow, and naturality of $u$,
  that is, the claim that the outer square commutes, follows from
  commutativity of the circuits involving $(g_J,x)$.  This finishes
  the proof that $j_S$ is an equivalence onto $(G\backsslash\mathfrak
  X)_S^{\mathrm{triv}}$.
\end{proof}

Now, given a stack $\mathfrak C$, applying $\Fun_{\mathfrak
  S}(\textrm{---},\mathfrak C)$ to diagram \eqref{eq:23} produces a
(weak) cosimplicial groupoid.  The following descent calculation for
$G$-stacks is then a corollary of proposition~\ref{prop:9}.

\begin{proposition}
  \label{prop:6}
  For any stack $\mathfrak C$ and $G$-stack $\mathfrak X$, diagram
  \eqref{eq:23} induces an equivalence of groupoids
  \begin{equation*}
    \Fun_{\mathfrak S}(G\backsslash\mathfrak X, \mathfrak C)
    \cong
    \lim \left(  \Fun_{\mathfrak S}(\mathfrak X,\mathfrak C)
      \rightrightarrows \Fun_{\mathfrak S}(G \times \mathfrak X,\mathfrak C)
      \threerightarrows \cdots \vphantom\fourrightarrows\right).
  \end{equation*}
\end{proposition}

Again, a concrete description of $2$-limits in the $2$-category of
small categories can be found in \textcite[chapter 5]{MR2229946}.  For
the convenience of the reader, we give a quick summary here.  We fix
the same notations as in the discussion of colimits above; in
particular, we have a diagram $F\colon D \to \mathrm{Cat}$.  Then (a
model for) the limit of $F$ is the category whose objects are (pseudo)
natural transformations $\Delta_{\mathrm{pt}} \to F$ with domain the
constant functor with value the discrete category with one object, and
whose morphisms are modifications between them.  In concrete terms, an
object consists of a collection of objects $a_i \in A_i$ for each $i
\in D$ together with isomorphisms $\tau_f\colon f(a_i) \to a_j$ for each
morphism $f\colon i \to j$ in $D$; these data are required to satisfy
certain coherence conditions.  A morphism $(a'_i, \tau_f') \to
(a_i,\tau_f)$ consists of a collection of morphisms $a_i' \to a_i$ in
$A_i$ for each $i \in D$, subject to appropriate conditions.

\begin{proposition}
  \label{prop:4}
  Given a homomorphism of sheaves of groups $h\colon G\ \to H$ and an
  $H$-stack $\mathfrak X$, we have an equivalence
  \begin{equation*}
    G \backsslash \mathfrak X
    \xrightarrow\cong
    G\backsslash \mathrm{pt}
    \times_{H\backsslash \mathrm{pt}}
    H \backsslash \mathfrak X.
  \end{equation*}
\end{proposition}

\begin{proof}
  The various maps of stacks involved in the statement of the
  proposition are induced by the obvious maps between the simplicial
  diagrams of which they are a colimit (cf.\
  proposition~\ref{prop:9}), as well as the universal property of the
  fiber product.

  It follows from proposition~\ref{prop:9} that $G\backsslash\mathfrak
  X$ is obtained by stackifying the prestack that assigns to $S \in
  \mathfrak S$ the groupoid whose objects are objects $x,y, \ldots \in
  \mathfrak X$, and morphisms $x \to y$ are pairs $(g, \xi)$, where $g
  \in G_S$ and $\xi\colon g\cdot x \to y$ is a morphism in $\mathfrak
  X_S$.  The formation of fiber products commutes with stackification
  (since the latter is built using limits and filtered colimits), so
  the codomain of our fibered functor has a similar description as the
  stackification of a fiber product of prestacks.
  
  Now, in terms of the above data, the map $G
  \backsslash\mathfrak X \to H \backsslash\mathfrak X$
  sends
  \begin{equation*}
    x \mapsto x,\quad (g, \xi) \mapsto (h(g), \xi)
  \end{equation*}
  while $G \backsslash\mathfrak X \to G \backsslash\mathrm{pt}$ sends
  \begin{equation*}
    x \mapsto \mathrm{pt},\quad (g, \xi) \mapsto g.
  \end{equation*}
  Thus it is clear that, at the level of prestacks, our fibered
  functor is fully faithful and essentially surjective, and the result
  follows.
\end{proof}

\section{Low-dimensional Euclidean supergeometry}
\label{sec:eucl-superm}

In the category $\mathrm{SM}$ of supermanifolds, $\mathbb R^{1\vert1}$
has a (noncommutative) group structure given by
\begin{equation*}
  \mathbb R^{1\vert1} \times \mathbb R^{1\vert1} \to \mathbb
  R^{1\vert1}, \quad ((t,\theta), (t',\theta')) \mapsto (t + t' +
  \theta\theta', \theta + \theta').
\end{equation*}
The Lie algebra of left-invariant vector fields is free on one odd
generator $D = \partial_\theta - \theta\partial_t$, and actions of
$\mathbb R^{1\vert1}$ correspond (bijectively, modulo noncompactness
issues) to odd vector fields.  Similarly, $\mathbb R^{0\vert1}$ has a
Lie algebra spanned by an odd element $\partial_\theta$ squaring to
$0$, and its actions correspond bijectively to homological vector
fields, i.e., those odd $Q$ such that $[Q, Q] = 0$.

The definition of Euclidean structures on supermanifolds follows the
philosophy of Felix Klein's Erlangen program.  One starts by fixing a
model space and a subgroup of diffeomorphisms, called the isometry
group; a Euclidean structure is then a maximal atlas whose transition
maps are isometries.  This idea is explained in detail in
\textcite[sections~2.5 and 4.2]{MR2742432}.  In (real) dimensions
$0\vert1$ and $1\vert1$, the model spaces are $\mathbb R^{0\vert1}$
respectively $\mathbb R^{1\vert1}$ with isometry groups
\begin{equation*}
  \operatorname{Isom}(\mathbb R^{0\vert1})
  = \mathbb R^{0\vert1} \rtimes \mathbb Z/2,
  \quad
  \operatorname{Isom}(\mathbb R^{1\vert1})
  = \mathbb R^{1\vert1}\rtimes\mathbb Z/2.
\end{equation*}
In both cases, $\mathbb Z/2$ acts by negating the odd coordinate and
$\mathbb R^{d\vert1}$ acts by \emph{left} multiplication (this choice
influences our sign conventions, and dictates whether to work with
left of right group actions at various places).

The differential form $d\theta \wedge d\theta$ on $\mathbb
R^{0\vert1}$ is invariant under isometries, and therefore determines a
canonical fiberwise $2$-form $\zeta$ on any family $\Sigma \to S$ of
Euclidean $0\vert1$-manifolds.  Conversely, any closed, nondegenerate,
even fiberwise $2$-form on $\Sigma$ is locally of this form and
determines a Euclidean structure.

In dimension $1\vert1$, Euclidean structures also admit ad hoc
definitions in terms of sections of certain sheaves.  In the remainder
of this section, we discuss some of those alternative definitions, and
study the stack of $1\vert1$-dimensional closed connected Euclidean
supermanifolds, which we will also call Euclidean supercircles.  This
appendix is a survey of material I learned from Stephan Stolz, some of
which does not seem to have appeared in the literature.

\subsection{Euclidean structures in dimension $1\vert1$}
\label{sec:eucl-struct-dimens-dim-1-1}

In \textcite[section~2.3]{MR2407109}, a conformal structure on a
$1\vert1$-manifold $X$ is defined to be a distribution $\mathcal D$
(i.e., a subsheaf of the tangent sheaf $T_X$) of rank $0\vert1$
fitting in a short exact sequence
\begin{equation}
  \label{eq:2}
  0 \to \mathcal D \to T_X \to \mathcal D^{\otimes 2} \to 0
\end{equation}
(see also \cite[lecture~3]{MR1707282}).  A Euclidean structure is then
defined to be a choice, up to sign, of an odd vector field $D$
generating $\mathcal D$.  The fundamental example is the vector field
$D = \partial_\theta - \theta \partial_t$ on $\mathbb R^{1\vert1}$.
Note that it squares to $-\partial_t$, so in fact $D$, $D^2$ generate
$T_{\mathbb R^{1\vert1}}$.  More generally, conformal and Euclidean
structures on a family $X \to S$ of $1\vert1$-manifolds are
appropriate splittings or sections of the vertical tangent bundle
$T_{X/S}$.

We want to show that this is equivalent to the original definition.
Denote by $\mathfrak E$ and $\mathfrak E'$ the stacks of families of
$1\vert1$-dimensional Euclidean manifolds according to the chart
definition respectively the vector field definition.  It is clear that
we have a map $\mathfrak E \to \mathfrak E'$, since the transition
maps of a Euclidean chart preserve the canonical vector field $D$ on
$\mathbb R^{1\vert1}$ up to sign.  Now, given an object in $\mathfrak
E'$, the atlas from proposition~\ref{prop:14} below is indeed
Euclidean, by propositions \ref{prop:8} and \ref{prop:15}.  This gives
an inverse map $\mathfrak E \to \mathfrak E'$.

\begin{proposition}
  \label{prop:8}
  The subgroup of diffeomorphisms of $\mathbb R^{1|1}$ preserving the
  form $\omega = dt - \theta d\theta$ is precisely
  $\operatorname{Isom}(\mathbb R^{1|1}) = \mathbb R^{1|1} \rtimes
  \mathbb Z/2$, acting in the standard way on the left.
\end{proposition}

A correct reading of this assertion requires that we think in
families; thus, the claim is that the subsheaf of $\Diff(\mathbb
R^{1\vert1}) \subset \underline{\mathrm{SM}}(\mathbb
R^{1\vert1},\mathbb R^{1\vert1})$ preserving $\omega$ is representable
by the Lie group $\mathbb R^{1\vert1} \rtimes \mathbb Z/2$.  Moreover,
it will be clear from the proof that the proposition is true locally
in $\mathbb R^{1\vert1}$, that is, if $U \subset\mathbb R^{1\vert1}$
is a connected domain, then the sheaf of embeddings $U \to \mathbb
R^{1\vert1}$ preserving $\omega$ is $\mathbb R^{1\vert1} \rtimes
\mathbb Z/2$.

\begin{proof}
  An $S$-family of diffeomorphisms of $\mathbb R^{1\vert 1}$ is given
  by a diffeomorphism
  \begin{equation*}
    \Phi\colon S \times \mathbb R^{1\vert 1} \to S
    \times \mathbb R^{1\vert 1}
  \end{equation*}
  commuting with the projections onto $S$.  We can express this
  diffeomorphism in terms of a map $\phi\colon S \times \mathbb
  R^{1\vert 1} \to \mathbb R^{1\vert 1}$ by the formula
  \begin{equation*}
    (s, x) \mapsto (s, \phi(s,x) \cdot x),
  \end{equation*}
  where $s$, $x$ should be interpreted as $T$-points of $S$ and
  $\mathbb R^{1\vert 1}$ for a generic supermanifold $T$, and $\cdot$
  indicates the usual group operation on $\mathbb R^{1\vert 1}$.
  Writing $\phi=(r,\eta) \in (\mathbb R \times \mathbb R^{0\vert
    1})_{S \times \mathbb R^{1\vert1}}$ and $x = (t, \theta) \in
  (\mathbb R \times \mathbb R^{0\vert 1})_T$ in terms of their
  components, the above formula becomes
  \begin{equation*}
    (s, t, \theta) \mapsto (s, t+r(s,t,\theta)+\eta(s,t,\theta)
    \theta, \eta(s,t,\theta) + \theta).
  \end{equation*}
  Hence the equation $\omega = \Phi^*\omega$ reads
  \begin{equation*}
    dt - \theta d\theta = dt +dr - \theta d\theta -
    (2\theta + \eta) d\eta.
  \end{equation*}
  To analyze the restrictions imposed by this equation, let us write
  \begin{equation*}
    r = r_0 + r_1\theta,\ \eta = \eta_1 +
    \eta_0\theta,\text{ where } r_i, \eta_i \in C^\infty(S\times
    \mathbb R)^i.
  \end{equation*} Then $dr - (2\theta + \eta)d\eta = 0$ gives us
  \begin{align}
    \label{eq:93} 0 = {} & dr_0 - \eta_1d\eta_1\\
    \label{eq:94} & + (dr_1 + (2+\eta_0) d\eta_1 - \eta_1d\eta_0)\theta\\
    \label{eq:95} & + (r_1 - \eta_1\eta_0) d\theta\\
    \label{eq:96} & - (2 + \eta_0)\eta_0 \theta d\theta.
  \end{align}
  Each individual line above vanishes.  From~\eqref{eq:96}, we get
  that $\eta_0 = 0$ or $-2$, since either $\eta_0$ or $(2+\eta_0)$ has
  nonzero reduced part and hence is invertible, and \eqref{eq:95}
  tells us that $r_1=\eta_0\eta_1$.  Plugging that into \eqref{eq:94},
  we get $(2+2\eta_0)d\eta_1=0$, so $d\eta_1=0$ since the factor in
  front of it is a nonzero constant.  Finally, \eqref{eq:93} implies
  that $dr_0=0$.

  Now, recall that those formulas should be interpreted as equalities
  of $S$-families of differential forms on $\mathbb R^{1\vert1}$,
  i.e., sections of $\Omega^*(S\times \mathbb R^{1\vert1})$ modulo
  $\Omega^{\geq 1}(S)$.  So in fact we have $r_0, \eta_1 \in
  C^\infty(S)$, and there is a locally constant function $a = 1 +
  \eta_0 \in (\mathbb Z/2)_S = \{\pm 1\}_S$.  Therefore the
  diffeomorphism $\Phi$ determines and is determined by $(r_0, \eta_1,
  a) \in (\mathbb R^{1\vert 1} \rtimes \mathbb Z/2)_S$ via the
  correspondence
  \begin{equation*} (r_0, \eta_1, a) \mapsto \phi_{r_0, \eta_1, a} =
    (r_0+(a-1)\eta_1\theta, \eta_1+(a-1)\theta) \in \mathbb
    R^{1\vert1}_{S\times \mathbb R^{1\vert1}}.
  \end{equation*}

  It is simple to check that any choice of $(r_0, \eta_1, a)$ as above
  determines a diffeomorphism preserving $\omega$, and that the
  choices $(r_0,\eta_1,1)$ respectively $(0,0,-1)$ act as translation
  by $(r_0,\eta_1)$ respectively negation of the odd variable.
  Therefore, to finish the proof, we just need to verify that given a
  second diffeomorphism $\Phi'$ prescribed, in a similar way, by
  $(r_0', \eta_1', a')$, the composition
  \begin{equation*}
    (s,t,\theta) \overset\Phi\mapsto 
    \phi_{r_0, \eta_1,a}(s,\theta) \cdot (s,t,\theta) \overset{\Phi'}\mapsto
    \phi_{r_0', \eta_1',a'}(s,\theta')\cdot \phi_{r_0, \eta_1,a}(s,\theta)\cdot (s,t,\theta),
  \end{equation*}
  where $\theta' = \eta+(a-1)\theta$ is the $\theta$-component of the
  middle term, agrees with the action of the product $(r_0',\eta_1',
  a') \cdot (r_0,\eta_1,a)$; more explicitly,
  \begin{equation*}
    \phi_{(r_0' + r_0 + a'\eta_1'\eta_1, \eta_1'+a'\eta_1,a'a)}
    (s,\theta)= \phi_{r_0',
      \eta_1',a}(s,\eta+(a-1)\theta)\cdot\phi_{r_0, \eta_1,a}(s,\theta).
  \end{equation*}
  This is a tedious but straightforward calculation.
\end{proof}

\begin{proposition}
  \label{prop:15}
  A diffeomorphism of $\mathbb R^{1\vert1}$ preserves $\omega = dt
  -\theta d\theta$ if and only if it preserves $D = \partial_\theta -
  \theta\partial_t$ up to sign.
\end{proposition}

\begin{proof}
  If an $S$-family of diffeomorphisms $\Phi\colon S \times \mathbb
  R^{1\vert1} \to S \times \mathbb R^{1\vert1}$ preserves $\omega$,
  then it is determined by $\phi \in (\mathbb R^{1\vert1} \rtimes
  \mathbb Z/2)_S$ and it is easy to check that it sends $D$ to either
  $D$ or $-D$.  Conversely, if $\Phi_*D = \pm D$, then $\Phi_* D^2 =
  (\pm D)^2$, so that
  \begin{equation*}
    \langle D, \Phi^*\omega\rangle = \langle \Phi_*D,\omega\rangle=0, 
    \quad \langle D^2, \Phi^*\omega \rangle = \langle D^2, \omega \rangle.
  \end{equation*}
  Since $D,D^2$ generate $T_{\mathbb R^{1\vert1}}$ as a
  $C^\infty_{\mathbb R^{1\vert1}}$-module, it follows that
  $\Phi^*\omega = \omega$.
\end{proof}

\begin{proposition}
  \label{prop:14}
  Let $X \to S$ be an $S$-family of $1\vert1$-manifolds and $D$ a
  vertical vector field generating a distribution as in \eqref{eq:2}.
  Then $X$ admits an atlas such that $D$ can be written locally as
  $\partial_\theta - \theta \partial_t$.
\end{proposition}

\begin{proof}
  We apply the Frobenius theorem \cite[lemma 3.5.2]{MR1701597} to the
  vector field $D^2$.  This gives us local charts $(t, \theta) \colon
  U \subset X \to S\times\mathbb R^{1\vert1}$ where $D^2$ gets
  identified with $-\partial_t$.  With respect to one of those charts,
  we can write
  \begin{equation*}
    D = f \partial_\theta + g\partial_t, \quad f = f_0+f_1\theta,
    \quad g = g_1+g_0\theta,
  \end{equation*}
  where $f_i,g_i \in C^\infty(S \times \mathbb R)^i$, so that
  \begin{equation*}
    D^2 = f(\partial_\theta f) \partial_\theta +
    f(\partial_\theta g) \partial_t + g(\partial_tf)\partial_\theta +
    g(\partial_t g) \partial_t
  \end{equation*}
  (the remaining terms one could expect in this expansion involve
  $\partial_\theta^2$, $g^2$, or $[\partial_\theta, \partial_t]$, so
  they vanish).  Inspecting the coefficients of $\partial_t$,
  $\theta \partial_t$, $\partial_\theta$, and $\theta\partial_\theta$
  respectively, we get
  \begin{align*}
    f_0g_0 + g_1g_1' = -1, &\quad f_1g_0 + g_1g_0' - g_0g_1' = 0,\\
    -f_0f_1 + g_1f_0' = 0, &\quad g_1f_1' + g_0f_0' = 0.
  \end{align*}
  The first equation implies that $f_0$, $g_0$ are invertible, and the
  fourth equation implies that $g_1g_0f_0' = 0$.  Multiplying the
  third equation by $g_0$ gives us $g_0f_0f_1 = 0$, so $f_1 = 0$.
  Using again the fourth equation, we conclude that $f_0' = 0$.
  Therefore (first equation), $g_0'$ is a multiple of $g_1$ and the
  second equation reduces to $g_0g_1' = 0 = g_1'$.  Finally, we learn
  from the first equation that $f_0$ and $-g_0$ are inverses.  To
  summarize, we have
  \begin{equation*}
    D = f_0\partial_\theta - f_0^{-1}\theta\partial_t, \text{ where }
    f_0\in C^\infty(S)^{\mathrm{even}}.
  \end{equation*}
  Performing the change of coordinates $t \mapsto t$, $\theta \mapsto
  f_0^{-1}\theta$, we can assume $f_0 = 1$, which finishes the proof.
\end{proof}

\subsection{Euclidean supercircles}
\label{sec:supercircles}

We are interested in the stack $\mathfrak K$ of closed connected
$1\vert1$-dimensional Euclidean manifolds.  Given a parameter
supermanifold $S$ and a map $l\colon S \to \mathbb R^{1\vert1}_{>0}$,
we can form the $S$-family of supercircles of length $l$, $K_l = (S
\times \mathbb R^{1\vert1})/\mathbb Z$, where the generator of the
$\mathbb Z$-action is described, in terms of $T$-points of $S \times
\mathbb R^{1\vert1}$, by $(s, u) \mapsto (s, l(s) \cdot u)$.
Moreover, given any map $r\colon S \to \mathbb R^{1\vert1}$, the
diffeomorphism of $S \times \mathbb R^{1\vert1}$, $(s,u)\mapsto (s,
r(s) \cdot u)$ descends to an isometry $K_{r^{-1}lr} \to K_l$, and the
flip $\operatorname{fl}\colon\mathbb R^{1\vert1} \to \mathbb
R^{1\vert1}$ (the diffeomorphism negating the odd coordinate) descends
to an isometry $K_{\operatorname{fl}(l)} \to K_l$, since $\mathrm{fl}$
is a group automorphism of $\mathbb R^{1\vert1}$.

We can assemble this collection of examples into a Lie groupoid as
follows.  Note that the right $\mathbb R^{1\vert1}$-action on itself
by conjugation extends to an action of the semidirect product $\mathbb
Z/2 \ltimes \mathbb R^{1\vert1}$ where $\mathbb Z/2$ acts via
$\mathrm{fl}$.  It is then clear that we have a map of stacks $\mathbb
R^{1\vert1}_{>0}\sslash(\mathbb Z/2 \ltimes \mathbb R^{1\vert1}) \to
\mathfrak K$.  To the $S$-point of the domain corresponding to a map
$l\colon S \to \mathbb R^{1\vert1}_{>0}$, it assigns $K_l$, and to the
morphism corresponding to the $S$-point $(a,r)$ of $\mathbb Z/2
\ltimes \mathbb R^{1\vert1}$, it assigns the isometry $K_{r^{-1}
  \mathrm{fl}^a(l) r} \to K_l$.  This only fails to be an equivalence
of stacks because the $S$-family of morphisms $(0, l)\colon l \to l$
in the domain stack maps to the identity map of $K_l$.
\begin{proposition}
  \label{prop:3}
  The fibered functor $\mathbb R^{1\vert1}_{>0}\sslash(\mathbb Z/2
  \ltimes \mathbb R^{1\vert1}) \to \mathfrak K$ is full and
  essentially surjective.
\end{proposition}

\begin{proof}
  Any isometry $K_{l'} \to K_l$ lifts to an isometry of their covers
  $S \times \mathbb R^{1\vert1} \to S \times \mathbb R^{1\vert1}$.  It
  follows from proposition~\ref{prop:8} that the isometry group of
  $\mathbb R^{1\vert1}$ is (no bigger than) $\mathbb R^{1\vert1}
  \rtimes \mathbb Z/2$, and this proves fullness.

  It remains to show that our fibered functor is essentially
  surjective.  Pick any $K \in \mathfrak K_S$.  Restricting to a
  neighborhood in $S$ if needed, fix a section $x\colon S \to K$ and a
  vector field $D_K$ specifying the Euclidean structure.  Then $D_K$
  gives us an action $\mu\colon \mathbb R^{1\vert1} \times K \to K$;
  composing with $x$, we get a map of $S$-families
  \begin{equation*}
    \mu_x\colon
    \mathbb R^{1\vert1}\times S \xrightarrow{\mathrm{id}\times x}
    \mathbb R^{1\vert1}\times K \overset\mu\to K.
  \end{equation*}
  Since the generators $D, D^2$ of the Lie algebra of $\mathbb
  R^{1\vert1}$ are $\mu$-related to the linearly independent vector
  fields $D_K, D_K^2$, $\mu_x$ is a local diffeomorphism.  Thus we can
  find a function $l\colon S \to \mathbb R^{1\vert1}_{>0}$ which is
  minimal, pointwise in $S$, among those $l$ such that $\mu(l, x) =
  x$.  Therefore $\mu_x$ factors through a diffeomorphism $K_l =
  (\mathbb R^{1\vert1} \times S)/\mathbb Zl \to K$.
\end{proof}

\begin{remark}
  At least locally, an $S$-family in $\mathfrak K$ is determined, up
  to isomorphism, by a conjugacy class in $(\mathbb R^{1\vert1}_{>0})_S$.
  However, an actual ``length'' function $l\colon S \to \mathbb
  R^{1\vert1}_{>0}$ is extra information, determined for instance by a
  basepoint (i.e., a section of the submersion $K \to S$).  In
  particular, the coarse moduli space of Euclidean supercircles is not
  a representable supermanifold.
\end{remark}

Each conjugation-invariant (generalized) submanifold of $\mathbb
R^{1\vert1}_{>0}$ gives rise to a full substack of $\mathbb
R^{1\vert1}_{>0}\sslash(\mathbb Z/2 \ltimes \mathbb R^{1\vert1})$, and
therefore to a full substack of $\mathfrak K$.  Here we are interested
in the choice $\{ 1 \} \subset \mathbb R^{1\vert1}_{>0}$, and we let
\begin{math}
  \mathfrak K_1 \subset \mathfrak K
\end{math}
denote the substack of supercircles of length $1$.  Recall also the
definition of $\mathfrak B^{\mathbb T} = \mathfrak B^{\mathbb
  T}(\mathrm{pt})$ from section~\ref{sec:t-equiv-bord}.

\begin{theorem}
  \label{thm:K1-versus-BT}
  There in an equivalence of stacks $\mathfrak K_1 \cong \mathfrak
  B^{\mathbb T}$.
\end{theorem}

\begin{proof}
  The fibered functor of proposition~\ref{prop:3} factors through an
  equivalence
  \begin{equation*}
    \mathrm{pt}\sslash(\mathbb Z/2 \ltimes \mathbb T^{1\vert1})
    \to \mathfrak K_1.
  \end{equation*}
  On the other hand, consider the $S$-point of $\mathfrak B^{\mathbb
    T}$ determined by the trivial family $\Sigma = S \times \mathbb
  R^{0\vert1}$ and the trivial bundle $P = \Sigma \times \mathbb T$
  with the standard connection $\omega = dt - \theta d\theta$.  It is
  easy to see, using proposition~\ref{prop:8}, that the automorphism
  group of $(P,\Sigma) \in \mathfrak B^{\mathbb T}_S$ is precisely
  $\Isom(\mathbb T^{1\vert1})$.  This determines a fully faithful
  fibered functor
  \begin{equation*}
    \mathrm{pt}\sslash(\mathbb Z/2 \ltimes \mathbb T^{1\vert1})
    \to \mathfrak B^{\mathbb T}.
  \end{equation*}
  It only remains to check that it is also essentially surjective.
  For contractible $S$ and any object of $\mathfrak B^{\mathbb T}_S$,
  we can assume the underlying bundles $\Sigma \to S$ and $P \to
  \Sigma$ are trivial.  So we just need to prove that the connection
  $\omega$ on $P$ can be taken to be the standard one.

  In general, a (fiberwise) connection on $P$ can be written as
  $\omega = dt + (f_1 + f_0 \theta) d\theta$ for functions $f_i \in
  C^\infty(S)$ of parity $i$.  The curvature condition imposes that
  $f_0 = -1$.  Under the gauge transformation of $P = S \times \mathbb
  R^{0\vert1} \times \mathbb T$ given by $(s, \theta, t) \mapsto (s,
  \theta, t - f_1(s)\theta)$, the connection $\omega$ pulls back to
  the standard $dt - \theta d\theta$.
\end{proof}

\printbibliography

\end{document}